\documentclass[hidelinks,onefignum,onetabnum]{siamart220329}

\usepackage{lipsum}
\usepackage{amsfonts}
\usepackage{graphicx}
\usepackage{epstopdf}
\usepackage{algorithmic}
\ifpdf
  \DeclareGraphicsExtensions{.eps,.pdf,.png,.jpg}
\else
  \DeclareGraphicsExtensions{.eps}
\fi

 \usepackage{amsmath,amssymb}
 \usepackage{mathrsfs}
 \usepackage{dsfont} 
 \usepackage{mathtools} 
 \usepackage{esint} 
 \usepackage{xfrac} 
 \usepackage{defs} 
 \usepackage{fonts} 

 \usepackage[shortlabels]{enumitem} 

\usepackage{cite} 
\usepackage{cleveref}

\newsiamremark{remark}{Remark}
\newsiamremark{hypothesis}{Hypothesis}
\crefname{hypothesis}{Hypothesis}{Hypotheses}
\newsiamthm{claim}{Claim}

\headers{Nonlocal problems with local boundary conditions I}{James M. Scott and Qiang Du}

\title{Nonlocal problems with local boundary conditions I: function spaces and variational principles
}

\author{James M. Scott\thanks{Department of Applied Physics and Applied Mathematics, and the Data Science Institute, Columbia University, New York, NY 10027 
  (\email{jms2555@columbia.edu}, \email{qd2125@columbia.edu}).}
\and Qiang Du\footnotemark[1]}

\usepackage{amsopn}

\begin{document}

\maketitle

\begin{abstract}
We present a systematic study on a class of nonlocal integral functionals
for functions defined on a bounded domain
and the naturally induced function spaces.
The function spaces are equipped with a seminorm depending on finite differences weighted by a position-dependent function, which leads to heterogeneous localization on the domain boundary.
We show the existence of minimizers for nonlocal variational problems with classically-defined, local boundary constraints,
together with the variational convergence of these functionals to classical counterparts in the localization limit. This program necessitates a thorough study of the nonlocal space; we demonstrate properties such as a Meyers-Serrin theorem, trace inequalities, and compact embeddings,  which are facilitated by new studies of boundary-localized convolution operators.
\end{abstract}

\begin{keywords}
nonlocal equations, boundary-value problems, nonlocal function spaces, fractional Sobolev spaces, Gamma convergence, heterogeneous localization, vanishing horizon
\end{keywords}

\begin{MSCcodes}
45K05, 35J20, 46E35
\end{MSCcodes}

\section{Introduction}

We are interested in nonlocal variational problems posed on a bounded domain $\Omega \subset \bbR^d$ with natural energy space characterized by the seminorm
\begin{equation}
    \int_{\Omega} \int_{\Omega} \gamma(\bx,\by) |u(\by)-u(\bx)|^p \, \rmd \by \, \rmd \bx\,,
\end{equation}
for measurable functions $u : \Omega \to \bbR^d$. 
Here, the constant $p \in [1,\infty)$ is a Lebesgue exponent and the function $\gamma$ represents a nonlocal interaction kernel.
In this work, our focus is given to kernels of the form
\begin{equation*}   \gamma(\bx,\by) =  
\mathds{1}_{ \{ |\by-\bx| < \delta \eta(\bx) \}}
\frac{C}{ |\by-\bx|^{\beta} (\delta \eta(\bx))^{d+p-\beta} }\,
\end{equation*}
with an exponent $\beta\in [0, d+p) $, a normalization constant $C>0$, a scaling parameter $\delta>0$,  and a position-dependent weight $\eta=\eta(\bx)$.

Variational problems associated to nonlocal energies with various forms of $\gamma(\bx,\by)$ appear widely in both analysis and applications \cite{andreu2010nonlocal,barles2014neumann,barlow2009non,braides2022compactness,bucur2016nonlocal,caffarelli2007extension,caffarelli2011regularity,craig2016blob,carrillo2019blob,Du12sirev,
Du2019book,Gilboa-Osher,grinfeld2005non,lovasz2012large,MeKl04,mogilner1999non,Nochetto2015,Coifman05,shi2017,Valdinoci:2009,Kassmann}.
Earlier studies of these variational problems on bounded domains have taken several different paths. 
Along the path that
$\gamma(\bx,\by) =C|\by-\bx|^{-\beta}
$ with
$\beta \in (0,d+p)$ and 
$\eta$ constant, so that ${\gamma=\gamma}(\bx,\by)$ is singular on the diagonal $\bx=\by$, both volume-constraint problems and classical boundary-value problems have been investigated, see for example \cite{cortazar,DyKa19,Ros16,andreu2010nonlocal} and additional references cited therein. If in particular $\beta > d + 1$, then classical boundary values can be prescribed via the trace operator, see \cite{A75,Hitchhiker,grube2023robust} for cases of singular kernels that give rise to solutions in fractional
Sobolev-Slobodeckij spaces. 

Down another path, with a compactly supported and translation invariant kernel, e.g., $\gamma(\bx,\by)= \delta^{-d-p}\rho(|\bx-\by|/\delta)$ for a function $\rho$ supported in the unit interval $(0,1)$ and a constant (horizon parameter $\delta>0$) that measures the range of nonlocal interactions. One natural route to take is to define the so-called nonlocal \textit{volumetric constraint} to complement the equation defined on $\Omega$ \cite{Du12sirev,Du-NonlocalCalculus,Du2022nonlocal}. An example is the prescription of $u(\bx)$ in a layer consisting of $\bx\in\Omega^c$ with $\dist(\bx, \Omega)<\delta$.
An alternative is to modify the nonlocal interaction rules involving $u=u(\bx)$ in a layered domain, say, for $\bx \in \Omega$ with $\dist(\bx, \p \Omega) < \delta$. These volumetric conditions can recover traditional boundary conditions in the local limit as $\delta\to 0$ under suitable conditions, see for example \cite{bellido2015hyperelasticity,mengesha2015localization,Du2022nonlocal,d2020physically,foss2022convergence,Lipton-2016}. Meanwhile, in the regime $\delta\to \infty$ with a suitably rescaled fractional kernel, these problems are related to studies of fractional differential equations defined on a bounded domain
\cite{bellido2021restricted,Foghem2022,Grubb,d2021connections}. In addition, one can find connections to the continuum limits of discrete graph operators and discrete particle interactions \cite{braides2022compactness,Coifman05,garcia2020error}.
{For various nonlocal problems, studies of their well-posedness subject to nonlocal volumetric constraints can be found, for example, in \cite{Du-NonlocalCalculus, MengeshaDuElasticity}, which offered desirable mathematical insight as demonstrated for a number of applications such as the peridynamics models developed in mechanics \cite{Silling2000,Silling2008,Du-Zhou2011,Lipton-2014,Valdinoci-peridynamic}, nonlocal diffusion and jump processes \cite{Du12sirev,Burch2014exit-time} and nonlocal Stokes equations for the analysis of smoothed particle hydrodynamics \cite{Du-Tian2020}.}

Still another path is to mix classical boundary conditions and volume-constraint conditions in constitutive models that blend local and nonlocal models. For an extensive discussion relating to the many choices of blended models in applications such as peridynamics, see the survey \cite{d2022review}.

We are interested in  boundary-value problems for nonlocal problems on a bounded domain in the classical sense, that is, the boundary conditions are prescribed on $\partial \Omega$ only.
The motivation is two-fold: first,  while the nonlocal constraints are natural, they are not perfect choices.  Theoretically, nonlocal constraints may raise unintended concerns about the regularity of solutions, for instance, non-constant functions vanishing in a layer of nonzero measure no longer enjoy analyticity, and solutions of problems with smooth kernels may experience non-physical or undesirable jumps at the boundary due to unmatched nonlocal constraints \cite{Du2022nonlocal}. In practice, developers of simulation codes for applications of nonlocal models
have ample practical reasons to keep local boundary conditions in implementation, even though a nonlocal model might be derived and/or deemed a better modeling choice in the domain of interest.

To allow for the prescription of local boundary conditions, the nonlocal energies and the nonlocal solution spaces must be defined so that boundary values of the solutions make sense. 
In the case where the
kernel $\gamma=\gamma(\bx,\by)$ does not have sufficient singularity on the diagonal  $\bx=\by$, it means that some localizing property near the boundary should hold. For instance,
in \cite{du2022fractional, tian2017trace, tao2019nonlocal},
a function $\delta \min\{1, \dist(\bx, \partial \Omega)\}$ is introduced to characterize the extent of nonlocal interactions at a point $\bx\in \Omega$, instead of taking a constant $\delta$ as the horizon parameter everywhere in the domain.
Clearly, the interactions are localized on the boundary.
A consequence of this type of {\em heterogeneous localization} is that functions in $L^p(\Omega)$  with a  bounded energy can have well-defined traces on $\partial \Omega$ to allow classical,  
local boundary conditions for nonlocal problems, see \cite{tian2017trace,Foss2021}. Studies of nonlocal operators with heterogeneous localization also appear in the seamless coupling of local and nonlocal models \cite{tao2019nonlocal}.

In this first part of a series of works on the analysis of these nonlocal variational problems with local boundary conditions imposed via heterogeneous localization, 
we rigorously establish their well-posedness theory and examine the convergence to their classical local counterparts.
We adopt a general heterogeneous localization strategy elucidated in later sections by a function $q$ controlling the rate of localization at the boundary, which bears significant consequences in the studies presented in subsequent papers.
This aspect is novel in the context of the analysis of nonlocal problems, so we also study the general nonlocal function spaces $\mathfrak{W}^{\beta,p}[\delta;q](\Omega)$.

\subsection{Nonlocal function spaces}\label{sec:FunctionSpaces}
Throughout the paper, we assume that for $d\geq 1$, $\Omega \subset \bbR^d$ is an open connected set (a domain) that is bounded and Lipschitz.
To describe our main findings, we first introduce the function space
\begin{equation}
    \frak{W}^{\beta,p}[\delta;q](\Omega) := \{ u \in L^p(\Omega) \, :\, [u]_{ \frak{W}^{\beta,p}[\delta;q](\Omega) } < \infty \}\,, 
\end{equation}
which is a Banach space equipped with the norm determined by
$$
\Vnorm{u}_{\frak{W}^{\beta,p}[\delta;q](\Omega)}^p := \Vnorm{u}_{L^p(\Omega)}^p + [u]_{\frak{W}^{\beta,p}[\delta;q](\Omega)}^p\, .
$$
The specific form of the nonlocal seminorm under consideration here, for given exponents $p$ and $\beta$ and constant $\delta$,
is defined by 
\begin{equation}\label{eq:Intro:NonlocalSeminorm}
[u]_{\frak{W}^{\beta,p}[\delta;q](\Omega)}^p =   \int_{\Omega} \int_{\Omega} \gamma_{\beta,p}[\delta;q](\bx,\by) |u(\by)-u(\bx)|^p \, \rmd \by \, \rmd \bx\,,
\end{equation}
where 
$p \in [1,\infty)$ and
\begin{equation}
\label{assump:beta}
 \beta \in [0,d+p)\,,
\tag{\ensuremath{\rmA_{\beta}}}
\end{equation}
taken as assumptions throughout the paper unless noted otherwise. The constant
$\delta > 0$ is the \textit{bulk horizon parameter} and
the kernel in \eqref{eq:Intro:NonlocalSeminorm} is defined as
\begin{equation}\label{eq:Intro:kernelgamma}
	\gamma_{\beta,p}[\delta;q](\bx,\by) := \mathds{1}_{ \{ |\by-\bx| < \delta q(\dist(\bx,\p \Omega)) \} } \frac{ C_{d,\beta,p} }{ |\bx-\by|^{\beta} } \frac{1}{ (\delta q(\dist(\bx,\p \Omega)))^{d+p-\beta} }\,.
\end{equation}
For a Lebesgue measurable set $A \subset \bbR^d$, $\mathds{1}_A$ defines its standard characteristic function. $C_{d,\beta,p} > 0$ is a normalization constant so that for any $\bx \in \Omega$, 
\begin{equation}\label{eq:Intro:StdKernelNormalization}
    \int_{\bbR^d} \gamma_{\beta,p}[\delta;q](\bx,\by)|\bx-\by|^p \, \rmd \by = \int_{B(0,1)} \frac{ C_{d,\beta,p} }{ |\bsxi|^{\beta-p} } \, \rmd \bsxi =  \frac{ \sqrt{\pi} \Gamma ( \frac{d+p}{2} ) }{ \Gamma(\frac{p+1}{2} ) \Gamma(\frac{d}{2}) } := \overline{C}_{d,p}\,,
\end{equation}
with $B(0,1)$ denoting the unit ball centered at the origin in $\mathbb{R}^d$ and $\Gamma(z)$ denoting the Euler gamma function.
In fact, we see directly that $C_{d,\beta,p} = \overline{C}_{d,p} \frac{d+p-\beta}{\sigma(\bbS^{d-1})}$, where $\sigma$ denotes surface measure and $\bbS^{d-1} \subset \bbR^d$ is the unit sphere. 
These constants are defined so that the nonlocal seminorm is consistent with the classical Sobolev seminorm in a precise way, as will be discussed later.

The function $q: [0,\infty) \to [0,\infty)$ is used to characterize the dependence of the localization  on the distance function. It is assumed to satisfy the following:
\begin{equation}\label{assump:NonlinearLocalization}	\begin{aligned}
		i) &\, q \in C^k([0,\infty)) \text{ for some } k \in \bbN \cup \{\infty\},
        \; q(0)=0 \text{ and } 0 < q(r) \leq r,\; \forall r >  0;\\
		ii) &\, 0 \leq q'(r) \leq 1,\; \forall r \geq 0\,, \text{ and for a fixed } c_q >0,\;  q'(r) > 0,\; \forall r \in (0,c_q]; \\
        iii) &\, \text{there exists } C_q \geq 1 \text{ such that } q(2r) \leq C_q \, q(r)\,\; \forall r \in (0,\infty); \\
		iv) &\, \text{if } k \geq 2\,, |q''(r)| \in L^{\infty}([0,\infty))\,.
	\end{aligned}
	\tag{\ensuremath{\rmA_{q}}}
\end{equation}
Conditions i) and ii) ensure that $q$ is (super)linear near $0$, and condition iii) is a kind of homogeneity condition. Condition iii) additionally implies that 
\begin{equation}\label{eq:NonlinLocAssump:Consequence}
    \frac{q(R)}{q(r)} \leq C_q \left( \frac{R}{r} \right)^{\log_2(C_q)} \quad \text{for all }  0 < r \leq R < \infty\,.
\end{equation}
Some examples of $q$ satisfying \eqref{assump:NonlinearLocalization} with $k = \infty$ are $q(r) = r$ and $q(r) = \frac{2}{\pi} \arctan(r)$. 
Another class of examples is $q(r) \approx \frac{1}{N} \min \{ r^N, 1 \}$ for $N \in \bbN$, mollified in a neighborhood of $r=1$ so that $q \in C^{k}$ for any desired $k$.

The function $\delta q(\dist(\bx,\p \Omega))$ used in \eqref{eq:Intro:kernelgamma} does not exceed $\delta$ for $\bx$ in all of $\Omega$, which leads to the naming of $\delta$ as the bulk horizon parameter, but shrinks to $0$ as $\bx\to \p\Omega$, hence leading to boundary localization. With these features, it
represents the extent of nonlocal interaction that takes on a more complex form than merely staying as a constant throughout the domain. The latter case, given by 
$\delta q(\dist(\bx,\p \Omega))=\delta$  
for any $\bx\in\Omega$ and a constant horizon parameter $\delta>0$ has been a popular choice for which the normalization condition for the kernel used in the seminorm implies that $[v]_{ \frak{W}^{\beta,p}[\delta;q](\Omega) } = \Vnorm{\grad v}_{L^p(\Omega)}$ for any linear function $v=v(\bx)$. 
Meanwhile, the choice of exponents $d+p-\beta$ and $\beta$ are made so that the nonlocal seminorm $[\cdot]_{ \frak{W}^{\beta,p}[\delta;q](\Omega) }$ serves as an analogue of the seminorm on the classical Sobolev space.  
Note that the only factors that ``genuinely'' determine the nonlocal function space $\frak{W}^{\beta,p}[\delta;q](\Omega)$ are $\beta$, $p$, $q$ and $\Omega$.  Different positive values of $\delta$ result in the same equivalent space, as demonstrated later in \Cref{thm:InvariantHorizon}.

Throughout this work we assume the existence of
a generalized distance function $\lambda : \overline{\Omega} \to [0,\infty)$ that satisfies the following:
\begin{equation}\label{assump:Localization}
	\begin{aligned}
		i) & \, \text{there exists a constant } \kappa_0 \geq 1 \text{ such that } \\
		&\quad \frac{1}{\kappa_0} \dist(\bx,\p \Omega) \leq \lambda(\bx) \leq \kappa_0 \dist(\bx,\p \Omega),\; \forall  \bx \in \overline{\Omega}\,;\\
        ii) & \, \text{there exists a constant } \kappa_1 > 0 \text{ such that } \\
        &\quad |\lambda(\bx) - \lambda(\by)| \leq \kappa_1 |\bx-\by|, \; \forall \bx,\by \in \Omega; \\
		iii) & \, \lambda \in C^0(\overline{\Omega}) \cap C^{k}(\Omega) \text{ for some } k \in \bbN_0 \cup \{\infty\}; \text{ and } \\
        iv) & \, \text{for each multi-index } \alpha \in \bbN^d_0 \text{ with } |\alpha| \leq k\,, \\
		&\quad \exists \kappa_{\alpha} > 0 \text{ such that } |D^\alpha \lambda(\bx)| \leq \kappa_{\alpha} |\dist(\bx,\p \Omega)|^{1-|\alpha|}, \;  \forall \bx \in \Omega\,.
	\end{aligned} \tag{\ensuremath{\rmA_{\lambda}}}
\end{equation}

Note that conditions i)-ii) are equivalent to conditions iii)-iv) when $k = 1$.
For any domain $\Omega$, a generalized distance function $\lambda$ with $k = \infty$ and all $\kappa_\alpha$ depending only on $d$ is guaranteed to exist; see \cite{Stein}.
Note that the distance function itself satisfies \eqref{assump:Localization} for $k = 0$ and $\kappa_1 = 1$, though in some of our later discussions, higher values of $k$ in  \eqref{assump:Localization} is preferred. Thus,
our analysis encompasses the case that $q(\dist(\bx,\p \Omega)$) is a smooth function that allows for specific forms of heterogeneous localization on the boundary $\p \Omega$, i.e., it is constant away from $\p \Omega$ and vanishes as $\bx$ approaches $\p \Omega$; see further discussion in \Cref{subsec:Examples}.

\subsection{Boundary-localized convolutions}\label{sec:LocalizedConvolution}
An essential tool in this analysis is the convolution-type operator
\begin{equation}\label{eq:ConvolutionOperator}
	K_{\delta}u (\bx) = K_\delta[\lambda,q,\psi](\bx) := \int_{\Omega} \frac{1}{(\eta_\delta[\lambda,q](\bx))^d} \psi \left( \frac{|\by-\bx|}{ \eta_\delta[\lambda,q](\bx) } \right) u(\by) \,\rmd \by , \; \bx \in \Omega.
\end{equation}
Here, $\psi:\bbR \to [0,\infty)$ is a standard mollifier satisfying
\begin{equation}\label{Assump:Kernel}
    \begin{gathered}
    \psi \in C^{k}(\bbR) \text{ for some } k \in \bbN_0 \cup \{\infty\}\,,
	\; \psi(x) \geq 0\, \text{ and } \, \psi(-x) = \psi(x),\;\forall x\in\bbR,\\
    [-c_\psi,c_{\psi}] \subset \supp \psi \Subset (-1,1) \text{ for fixed } c_{\psi} > 0\,, \; \text{ and }
    \int_{\bbR^d} \psi(|\bx|) \, \rmd \bx = 1\,.
    \end{gathered}
    \tag{\ensuremath{\rmA_{\psi}}}
\end{equation}
Meanwhile, the function
$\eta_\delta[\lambda,q](\bx) = q(\lambda(\bx))$ 
is given by 
\begin{equation}\label{eq:localizationfunction}
	\eta_\delta[\lambda,q](\bx) := \delta \eta[\lambda,q](\bx) := \delta 
 q(\lambda(\bx))\,, \quad\forall \bx \in \Omega\,,
\end{equation}
where 
$\eta_{1}[\lambda,q] = \eta[\lambda,q]$ is named a
\textit{heterogeneous localization function}. 
While we introduce these notations to emphasize the dependence on $q$ and $\lambda$ whenever multiple heterogeneous localization functions appear simultaneously in the same context, we will write $\eta_\delta[\lambda,q]$ simply as $\eta_\delta$ (with $\eta_1 = \eta$) whenever the dependence is clear from context. The same convention is applied to abbreviate 
$K_\delta[\lambda,q,\psi]$
as $K_\delta$.

For the study of the variational problems, the maximum admissible value of the bulk horizon parameter $\delta$ is chosen to depend on $\eta(\bx)$ as follows:
\begin{equation}
\begin{gathered} 
\delta \in (0, \min\{\underline{\delta}_0,\bar{\delta}_0\}) \;\text{ where }\;
\underline{\delta}_0 := \frac{1}{3 \max \{ 1, \kappa_1, C_q \kappa_0^{\log_2(C_q)} \} }\;\text{ and $\bar{\delta}_0$ is the}\\ 
\text {smallest positive root of } 
M_q(\delta) = \frac{1}{3} \,\text{ for }\,
M_q(\delta) := \frac{1+\kappa_1 \delta}{(1-\kappa_1 \delta)^2} \delta\,.
\label{eq:HorizonThreshold2}
\tag{\ensuremath{\rmA_{\delta}}}
\end{gathered}
\end{equation}
The precise definitions will be motivated later, but for now we note that by 
\eqref{eq:NonlinLocAssump:Consequence} and \eqref{assump:Localization}, we are guaranteed that for all $\delta < \underline{\delta}_0$
\begin{equation}
\begin{gathered}
    \eta_\delta[\lambda,q](\bx) \leq \delta C_q \kappa_0^{\log_2(C_q)} q(\dist(\bx,\p \Omega)) \leq \frac{1}{3} q(\dist(\bx,\p \Omega)) \text{ for all } \bx \in \Omega\,, \text{ and } \\
    |\eta_\delta[\lambda,q](\bx)-\eta_\delta[\lambda,q](\by)| \leq \frac{1}{3} |\bx-\by| \text{ for all } \bx, \by \in \Omega\,.
\end{gathered}   \label{eq:h0property}
\end{equation}

We refer to $K_{\delta}$ as a \textit{boundary-localized convolution} operator. This operator has all of the smoothing properties of classical convolution operators, and additionally recovers the boundary values of a function. To be precise, for all functions $u \in C^0(\overline{\Omega})$, $T K_\delta u = T u$, where $T u = u|_{\p \Omega}$ denotes the trace operator. This property of the boundary-localized convolution is preserved when the operator $T$ is extended to more general Sobolev and nonlocal function spaces.

Throughout the paper, the functions $\lambda$, $q$ and $\psi$ 
may have different orders of smoothness, and subscripts will be added for emphasis and the value of the index $k$ in 
\eqref{assump:NonlinearLocalization},
\eqref{assump:Localization}, and
\eqref{Assump:Kernel}
will vary, and will be specified in each context. 
For example, $q$ is assumed to satisfy \eqref{assump:NonlinearLocalization} for $k = k_q \geq 2$ in  \cref{thm:LocLimit:Dirichlet},
\cref{thm:LocLimit:Neumann} and \cref{thm:LocLimit:Robin}, while
$\psi$ is assumed to satisfy \eqref{Assump:Kernel} for $k=k_\psi \geq 1$ to get the estimate \eqref{eq:Intro:ConvEst:Deriv} in \cref{thm:convolution-estimate}, which relies on estimates in \cref{thm:gradientpsi} and is needed for 
\cref{thm:WellPosedness:Dirichlet},
\cref{thm:WellPosedness:Neumann}, \cref{thm:WellPosedness:Robin} and the theorems on the local limits.

Operators with similar boundary-localizing properties were first -- to our knowledge -- studied in \cite{burenkov1998sobolev,burenkov1982mollifying},
and later in \cite{hintermuller2020variable}.
However, in previous studies,  $\lambda$ is comparable either to $\dist(\bx,\p \Omega)$ or $\rme^{-1/\dist(\bx,\p \Omega)}$, with fixed bulk horizon $=1$; $\lambda$ does not involve the composition with more general nonlinearity $q$, or general bulk horizon $< 1$.
Thus, results like the boundedness on classical function spaces, etc. already obtained in those works, have more straightforward proofs in our setting; 
at the same time the operator $K_\delta$ takes on a form distinct from the earlier works, so that the related results are more general.
It is for this reason that the studies of the operators in classical Sobolev spaces are included in this work. Naturally, our chief interest is to develop their new properties associated to the nonlocal function space $\frak{W}^{\beta,p}[\delta;q](\Omega)$.

Our first main result concerns the utility of boundary-localized convolutions in the study of nonlocal function spaces and variational problems. To illustrate, we present the following theorem:
\begin{theorem}\label{thm:convolution-estimate}
    Let $K_\delta$ be as in \eqref{eq:ConvolutionOperator}, with all of the above assumptions. Then there exists a constant $C$ depending only on $d$, $\beta$, $p$, $\psi$, $q$, $\kappa_0$, $\kappa_1$ and $\Omega$ such that for all $\delta < \underline{\delta}_0$
 	\begin{equation}\label{eq:KdeltaError}
  	\Vnorm{u - K_{\delta} u }_{L^p(\Omega)} \leq 
    C \delta q(\diam(\Omega)) [u]_{\frak{W}^{\beta,p}[\delta;q](\Omega)}\,,
 	\end{equation}
    for all $u \in \frak{W}^{\beta,p}[\delta;q](\Omega)$. Further, for all $\delta$ satisfying \eqref{eq:HorizonThreshold2}
    \begin{equation}\label{eq:Intro:KdeltaError:Frac}
        [u - K_\delta u]_{W^{ (\beta-d)/p,p }(\Omega)} \leq C (\delta q(\diam(\Omega)) )^{ 1 - \frac{\beta-d}{p} } [u]_{ \frak{W}^{\beta,p}[\delta;q](\Omega) }
    \end{equation}
    for all $u \in \frak{W}^{\beta,p}[\delta;q](\Omega)$, whenever $\beta > d$.
    If in addition \eqref{Assump:Kernel} is satisfied for $k=k_\psi \geq 1$, then for all $\delta < \underline{\delta}_0$
    \begin{equation}\label{eq:Intro:ConvEst:Deriv}
		\begin{split}
			\Vnorm{ K_{\delta} u }_{W^{1,p}(\Omega)} 
			\leq C \Vnorm{u}_{\frak{W}^{\beta,p}[\delta;q](\Omega)}\,,
            \qquad \forall u \in \frak{W}^{\beta,p}[\delta;q](\Omega).
		\end{split}
	\end{equation}
    
\end{theorem}
The proofs are contained in \Cref{sec:NonlocalFxnSpEst}. The estimate \eqref{eq:Intro:ConvEst:Deriv} suggests that 
the nonlocal space $\frak{W}^{\beta,p}[\delta;q](\Omega)$, instead of other classical function spaces studied in the literature, is the natural function space on which results for the classical Sobolev space $W^{1,p}$ can be applied to the boundary localized convolution $K_\delta u$.
Meanwhile, \eqref{eq:KdeltaError} and \eqref{eq:Intro:KdeltaError:Frac} suggest that the nonlocal seminorm quantifies how $K_\delta u$ can be exchanged for $u$ in the $L^p$ or fractional Sobolev space $W^{(\beta-d)/p,p}$.
Indeed, the following two theorems, in addition to the Poincar\'e inequalities of \Cref{subsec:Poincare}, are 
proved partly as a consequence of the corresponding results for $W^{1,p}(\Omega)$ applied to $K_\delta$. See \Cref{sec:NonlocalSpaceFundProp} for the relevant assumptions and proofs.

\begin{theorem}[Density of smooth functions in the nonlocal space]\label{thm:Density}
    $C^{k}(\overline{\Omega})$ is dense in $\frak{W}^{\beta,p}[\delta;q](\Omega)$ for any $k \leq k_q$.
\end{theorem}

\begin{theorem}[Nonlocal trace theorem]\label{thm:TraceTheorem}
    Let $T$ denote the trace operator, i.e. for $u \in C^{1}(\overline{\Omega})$,
	  $ T u = u \big|_{\p \Omega}$.
	Then for each $\delta < \underline{\delta}_0$ the trace operator extends to a bounded linear operator $T : \frak{W}^{\beta,p}[\delta;q](\Omega) \to W^{1-1/p,p}(\p \Omega)$. Moreover
    there exists $C = C(d,p,\beta,q,\Omega)$ such that
	$$
	\Vnorm{Tu}_{W^{1-1/p,p}(\p \Omega)} \leq C \Vnorm{u}_{\frak{W}^{\beta,p}[\delta;q](\Omega)} ,\qquad\forall u \in \frak{W}^{\beta,p}[\delta;q](\Omega)\,.
	$$
\end{theorem}

The space $\frak{W}^{0,p}[\delta,id](\Omega)$ coincides with special cases of spaces considered in \cite{tian2017trace,Foss2021,du2022fractional}. Trace theorems were established in \cite{tian2017trace,du2022fractional} by analyzing nonlocal analogues of tangential and normal derivatives. 
In \cite{Foss2021}, it was shown that a specific boundary-localized convolution with $\delta = 1/3$ and $\psi(t) = \mathds{1}_{(-1/2,1/2)}(t)$ converges to a trace operator as $\bx \to \p \Omega$ for very wide classes of domains and functions. 
Here, we show that such a result can be obtained for Lipschitz domains and for a class of function spaces along another branch of generality;
the new method used in this work not only provides an alternative and more direct proof 
but also allows us to extend to the case of general $\beta$ and $q$ by using systematically-defined boundary-localized convolutions.
In addition, the density of smooth functions, the nonlocal Poincar\'e inequalities of \Cref{thm:PoincareNeumann} and \Cref{thm:PoincareRobin}, and the $L^p$-compactness  in the localizing limit of \Cref{thm:Compactness} are novel even for the spaces $\frak{W}^{0,p}[\delta,id](\Omega)$.

\subsection{Nonlocal energy functional}\label{sec:NonlocalEnergy}
The study of the nonlocal function space allows us to treat a wealth of variational problems with a common program; for the sake of clarity we will illustrate just a few in this work.

In order to treat general variational problems, we introduce a nonlocal kernel $\rho$ 
where $\rho : \bbR \to [0,\infty)$ satisfies
\begin{equation}\label{assump:VarProb:Kernel}
    \begin{gathered}
    \rho \in L^{\infty}(\bbR)\,, \quad 
    [-c_\rho,c_{\rho}] \subset \supp \rho \Subset (-1,1) \text{ for a fixed constant } c_{\rho} > 0\,.\\
    \;\text{ Moreover, } \rho(-x) = \rho(x)\, \text{and}\, \rho(x)\geq 0, \; \forall x\in  (-1,1)\,.
    \end{gathered}
    \tag{\ensuremath{\rmA_{\rho}}}
\end{equation}
We let $\Phi : [0,\infty) \to \bbR$ be a nonnegative and convex function that satisfies, for some $p > 1$, the $p$-growth condition for positive constants $c$ and $C$, that is,
\begin{equation}\label{eq:assump:Phi}
\Phi \text{ is convex,\ and }\,
	\max \{ 0, c ( |t|^p - 1 ) \} \leq \Phi(t) \leq C (|t|^p + 1) \,, \; \forall \, t\geq 0\,.
 \tag{\ensuremath{A_\Phi}}
\end{equation}

The general form of the nonlocal energy is then given by
\begin{equation}\label{eq:Intro:Varprob:Energy}
    \cE_{\delta}(u) := \int_{\Omega} \int_{\Omega} \rho \left( \frac{|\by-\bx|}{\eta_\delta(\bx)} \right) \frac{\Phi( \frac{ |u(\bx)-u(\by)|}{|\bx-\by|} )}{|\bx-\by|^{\beta-p} \eta_\delta(\bx)^{d+p-\beta} } \, \rmd \by \, \rmd \bx\,,
\end{equation}
where additionally the assumptions \eqref{assump:beta}, \eqref{assump:NonlinearLocalization}, and \eqref{eq:HorizonThreshold2} are adopted.
The nonlocal function space $\frak{W}^{\beta,p}[\delta;q](\Omega)$ is the natural choice of energy space for $\cE_\delta$, since the nonlocal seminorm remains the same under perturbations of
the heterogeneous localization $\lambda$ and kernel $\rho$; see \Cref{thm:InvariantHorizon} below.

The form of the functional we consider has principal part $\cE_\delta$, and is defined as
\begin{equation}\label{eq:Fxnal}
    \cF_\delta(u) := \cE_\delta(u) + \cG(K_\delta[\bar{\lambda},q,\psi] u) + \wt{\cG}_\delta(u) + \cG_{\beta>d}(u)\,,
\end{equation}
where we assume that
$\psi$ satisfies \eqref{Assump:Kernel} for $k_\psi \geq 1$ and $\bar{\lambda}$ satisfies \eqref{assump:Localization}.
Note that $\bar{\lambda}$ may not necessarily be equal to the $\lambda$ used for 
$\eta_\delta = \eta_\delta[\lambda,q]$ in the nonlocal functional given in 
\eqref{eq:Intro:Varprob:Energy}.

The functionals $\cG$, $\wt{\cG}_\delta$, and $\cG_{\beta>d}$ act as ``lower-order'' terms, and we explain the assumptions and significance of each of them in turn.
First, the functional $\cG : W^{1,p}(\Omega) \to \bbR$ is $W^{1,p}(\Omega)$-weakly lower semicontinuous, and satisfies, for some constants $c >0 $, $C >0$, $\theta \in (0,1)$ and $\Theta > 0$,
\begin{equation}\label{eq:LowerOrderTerm}
\begin{gathered}
    - c (1+\Vnorm{u}_{W^{1,p}(\Omega)}^{\theta p}) \leq \cG(u) \leq C ( 1 + \Vnorm{u}_{W^{1,p}(\Omega)}^{\Theta p})\,.
\end{gathered}
\end{equation}
The term $\cG(K_\delta[\bar{\lambda},q,\psi]u)$ is well-defined for $u \in \frak{W}^{\beta,p}[\delta;q](\Omega)$ thanks to the estimate \eqref{eq:Intro:ConvEst:Deriv}.
By introducing the convolution $K_\delta$, the term $\cG$ allows us to consider lower-order terms that, without mollification, may not be continuous in the nonlocal function space.

Next, we take $\wt{\cG}_\delta : \frak{W}^{\beta,p}[\delta;q](\Omega) \to \bbR$ to be a $\frak{W}^{\beta,p}[\delta;q](\Omega)$-weakly lower semicontinuous functional that satisfies, for some $\theta \in (0,1)$ and $\Theta >0$
\begin{equation}\label{eq:LowerOrderTerm3}
\begin{gathered}
    - c_\delta (1+\Vnorm{u}_{\frak{W}^{\beta,p}[\delta;q](\Omega)}^{\theta p}) \leq \wt{\cG}_\delta(u) \leq C_\delta ( 1 + \Vnorm{u}_{\frak{W}^{\beta,p}[\delta;q](\Omega)}^{\Theta p})\,,
\end{gathered}
\end{equation}
for all $\delta$ satisfying \eqref{eq:HorizonThreshold2}, analogous to the condition \eqref{eq:LowerOrderTerm}. Observe that the constants $c$ and $C$ in this case may in general depend on $\delta$.
The continuity conditions are more strict on $\wt{\cG}$ compared to $\cG$ because it is evaluated at $u$ itself, rather than the convolution.

Last, for $\beta > d$, the functional $\cG_{\beta>d}$ is chosen to take advantage of the continuous  embedding $\frak{W}^{\beta,p}[\delta;q](\Omega) \hookrightarrow W^{(\beta-d)/p,p}(\Omega)$; see \Cref{thm:Embedding:Fractional}.
For the ease of presentation, we note here that $\cG_{\beta>d}$ 
satisfies strong continuity properties contained in \eqref{eq:LowerOrderTerm2} in this case,  while it  is identically zero for $\beta \leq d$.

\subsection{Nonlocal Variational Problems}\label{sec:varprobs}
Let $\cF_\delta$ be defined as in \eqref{eq:Fxnal} with all the associated assumptions.
The first nonlocal problem we treat is one with an inhomogeneous Dirichlet-type constraint on $\p \Omega$, or more generally $\p \Omega_D$, a $\sigma$-measurable subset of $\p \Omega$ with a positive measure $\sigma(\p \Omega_D) > 0$.
Let $g \in W^{1-1/p,p}(\p \Omega_D)$, and define the set 
$$
\frak{W}^{\beta,p}_{g,\p \Omega_D}[\delta;q](\Omega) := \{ u \in \frak{W}^{\beta,p}[\delta;q](\Omega) \, : \, u = g \text{ on } \p \Omega_D \text{ in the trace sense } \}\,.
$$
Then we have the following:
\begin{theorem}\label{thm:WellPosedness:Dirichlet}
    There exists a function $u \in \frak{W}^{\beta,p}_{g,\p \Omega_D}[\delta;q](\Omega)$ satisfying
    \begin{equation}\label{eq:MinProb:Dirichlet}
        \cF_\delta(u) = \min_{ v \in \frak{W}^{\beta,p}_{g, \p \Omega_D} [\delta;q](\Omega) } \cF_\delta(v)\,.
    \end{equation}
\end{theorem}

A special case is when $g \equiv 0$ on $\p \Omega_D$, for which we consider the Banach space
\begin{equation}\label{eq:HomNonlocSpDef}
	\begin{split}
		\frak{W}^{\beta,p}_{0,\p \Omega_D}[\delta;q](\Omega) := \{ \text{closure of } C^{1}_c(\overline{\Omega} \setminus \p \Omega_D) \text{ with respect to } \Vnorm{\cdot}_{\frak{W}^{\beta,p}[\delta;q](\Omega)} \}\,.
	\end{split}
\end{equation}
We accordingly denote the Banach space $W^{1,p}_{0,\p \Omega_D}(\Omega)$ as the closure of $C^1_c(\overline{\Omega}\setminus \p \Omega_D)$ with respect to $\Vnorm{\cdot}_{W^{1,p}(\Omega)}$.
Then we can relax the assumptions on $\cG$ and $\wt{\cG}$ and still obtain existence:
\begin{theorem}\label{thm:WellPosedness:Dirichlet:Hom}
    Suppose that $\cG : W^{1,p}_{0,\p \Omega_D}(\Omega) \to \bbR$ is $W^{1,p}_{0,\p \Omega_D}(\Omega)$-weakly lower semicontinuous and satisfies \eqref{eq:LowerOrderTerm}, and suppose that $\wt{\cG} : \frak{W}^{\beta,p}_{0,\p \Omega_D}[\delta;q](\Omega) \to \bbR$ is $\frak{W}^{\beta,p}_{0,\p \Omega_D}[\delta;q](\Omega)$-weakly lower semicontinuous and satisfies \eqref{eq:LowerOrderTerm3}. Then there exists a function $u \in \frak{W}^{\beta,p}_{0,\p \Omega_D}[\delta;q](\Omega)$ satisfying
    \begin{equation}\label{eq:MinProb:Dirichlet:Hom}
        \cF_\delta(u) = \min_{ v \in \frak{W}^{\beta,p}_{0, \p \Omega_D} [\delta;q](\Omega) } \cF_\delta(v)\,.
    \end{equation}
\end{theorem}

The functional $\cF_\delta$ also has a minimizer in the nonlocal space
$$
\mathring{\frak{W}}^{\beta,p}[\delta;q](\Omega) := \{ u \in \frak{W}^{\beta,p}[\delta;q](\Omega) \, : \, (u)_\Omega = 0 \}\,,
$$
where $(u)_\Omega = \frac{1}{|\Omega|} \int_{\Omega} u(\bx) \, \rmd \bx = \fint_{\Omega} u(\bx) \, \rmd \bx $ denotes the integral average of $u$ over $\Omega$.

\begin{theorem}\label{thm:WellPosedness:Neumann}
    There exists a function $u \in \mathring{\frak{W}}^{\beta,p}[\delta;q](\Omega)$ satisfying
    \begin{equation}\label{eq:MinProb:Neumann}
        \cF_\delta(u) = \min_{ v \in \mathring{\frak{W}}^{\beta,p}[\delta;q](\Omega) } \cF_\delta(v)\,.
    \end{equation}
\end{theorem}

The final type of nonlocal problem is one with a Robin-type constraint.

For $b \in L^{\infty}(\p \Omega)$, define the functional
\begin{equation}\label{eq:Fxnal:Robin}
    \begin{split}
    \cF_\delta^R(u) &:= \cF_\delta(u) + \int_{\p \Omega} b |T u|^p \, \rmd \sigma \,.
    \end{split}
\end{equation}

\begin{theorem}\label{thm:WellPosedness:Robin}
Assume there exists a $\sigma$-measurable set $\p \Omega_R$ of $\p \Omega$ satisfying $\sigma(\p \Omega_R) > 0$ and $b(\bx) \geq b_0 > 0$ for a constant $b_0$ and
  $\sigma$-almost every $\bx \in \p \Omega_R$. Then
    there exists a function $u \in \frak{W}^{\beta,p}[\delta;q](\Omega)$ satisfying
    \begin{equation}\label{eq:MinProb:Robin}
        \cF_\delta^R(u) = \min_{ v \in \frak{W}^{\beta,p}[\delta;q](\Omega) } \cF_\delta^R(v)\,.
    \end{equation}
\end{theorem}

\begin{remark}\label{rmk:unique}
    In each of the cases, 
    if the non-principal terms of the nonlocal functionals (that is, the terms not equal to $\cE_\delta$) are all convex, then the minimizer obtained is unique. This is achieved using the strict convexity of $\cE_\delta$ via a standard argument.
\end{remark}

\subsection{Local limit}\label{sec:locallimit}
Inheriting the assumptions made in \cref{sec:varprobs}, we now present the second set of main results, which are on the localization limit. As $\delta \to 0$, we show that minimizers of nonlocal variational problems considered in \cref{sec:varprobs}
converge to a minimizer of a local functional with the principal part
\begin{equation}\label{eq:LocalizedEnergyDefn}
    \begin{gathered}
    \cE_0(u) := \bar{\rho}_{p,\beta} \int_\Omega \fint_{\bbS^{d-1}} 
    \Phi(|\grad u(\bx) \cdot \bsomega|) \, \rmd \sigma(\bsomega) \, \rmd \bx\,, \\
    \text{ where } \bar{\rho}_{p,\beta} := \int_{B(0,1)} |\bz|^{p-\beta} \rho(|\bz|) \, \rmd \bz\,.
    \end{gathered}
\end{equation}

This result is in the same spirit as the program carried out in \cite{ponce2004new, mengesha2015VariationalLimit}, in which nonlocal models are shown to be consistent with appropriate classical counterparts. 
Central to this analysis is the following result, 
which is coined the \textit{asymptotic} compact embedding,  in the asymptotic limit that the bulk horizon parameter $\delta \to 0$.
\begin{theorem}\label{thm:Intro:Compactness}
    For $p>1$, let $\{\delta_n \}_{n \in \bbN}$ be a sequence that converges to $0$, and let $\{ u_\delta \}_\delta \subset \frak{W}^{\beta,p}[\delta;q](\Omega)$ be a sequence such that $\sup_{\delta > 0} \Vnorm{u_\delta}_{L^p(\Omega)} \leq C < \infty$, and that 
    $$
    \sup_{\delta > 0} [u_\delta]_{\frak{W}^{\beta,p}[\delta;q](\Omega)} := B < \infty\,.
    $$
    Then $\{ u_\delta \}_\delta$ is precompact in the strong topology of $L^p(\Omega)$. Moreover, any limit point $u$ belongs to $W^{1,p}(\Omega)$ with $\Vnorm{\grad u}_{L^{p}(\Omega)} \leq B$.
\end{theorem}

A more general theorem and its proof are in \Cref{sec:Compactness}. One main idea of the proof is to use the boundary-localized convolution, leveraging the $L^p$ compact embedding for the space $W^{1,p}(\Omega)$ with the estimate \eqref{eq:Intro:ConvEst:Deriv}. 
This approach is novel in the context of nonlocal function spaces of this type, where typically compactness results are proven via estimates away from the boundary. However, in this case, since the boundary information is already contained in the convolution, no such estimates are needed.

Thanks to the properties of the nonlocal function space, the lower-order terms can be treated in the local limit; we only require additional continuity properties and a stricter assumption on the functional $\wt{\cG}_\delta$. To be precise, we assume that
\begin{equation}\label{eq:LocalLimit:LOTAssump}
    \begin{gathered}
        \cG \text{ is } W^{1,p}(\Omega)\text{-weakly continuous}\,, \text{ and that } \\
        \wt{\cG}_\delta = \wt{\cG} \text{ for all } \delta \text{ satisfying \eqref{eq:HorizonThreshold2}, where } \\
        \exists \tilde{m} \in [1,p] \text{ such that } \wt{\cG} : L^{\tilde{m}}(\Omega) \to \bbR \text{ is } L^{\tilde{m}}(\Omega)\text{-weakly continuous, with } \\
        - c (1+\Vnorm{u}_{L^{\tilde{m}}(\Omega)}^{\theta p}) \leq \wt{\cG}(u) \leq C ( 1 + \Vnorm{u}_{L^{\tilde{m}}(\Omega)}^{\Theta p})\,.
    \end{gathered}
\end{equation}
Note that any functional $\wt{\cG}$ defined on $L^{\tilde{m}}(\Omega)$ that is additionally $\frak{W}^{\beta,p}[\delta;q](\Omega)$-weakly continuous for all $\delta$ and satisfying \eqref{eq:LowerOrderTerm3} also satisfies \eqref{eq:LocalLimit:LOTAssump}.

The functional $\cG$ is permitted to satisfy much weaker conditions than either $\wt{\cG}$ or $\cG_{\beta>d}$.
Indeed, the convolution $K_\delta u$ approximates $u$ as $\delta \to 0$, so a very wide variety of lower-order terms, admissible typically only in the local case, can be considered in the nonlocal problem via this approximation.

With this compactness result in hand we can, under the additional assumptions that
\begin{equation}\label{eq:assump:rho:q}
\text{$\rho$ is nonincreasing on $[0,\infty)$ and  
    $q$ satisfies \eqref{assump:NonlinearLocalization} for  $k_q \geq 2$,}     
\end{equation}
obtain via $\Gamma$-convergence the following convergence of minima for each of the nonlocal problems.

\begin{theorem}\label{thm:LocLimit:Dirichlet}
    Assume \eqref{eq:LocalLimit:LOTAssump} and \eqref{eq:assump:rho:q}. For a sequence $\delta \to 0$, let $u_\delta \in \frak{W}^{\beta,p}_{g, \p \Omega_D}[\delta;q](\Omega)$ be a function satisfying \eqref{eq:MinProb:Dirichlet}.
    Then $\{u_\delta\}_\delta$ is precompact in the strong topology on $L^p(\Omega)$. Furthermore, 
    any limit point $u$ satisfies
    $u \in W^{1,p}_{g, \p \Omega_D}(\Omega)$, 
    where 
    $W^{1,p}_{g, \p \Omega_D}(\Omega) := \{ v \in W^{1,p}(\Omega) \, : \, Tv = g \text{ on } \p \Omega_D \},
    $
    and
    \begin{equation*}
        \cF_0(u) = \min_{ v \in W^{1,p}_{g, \p \Omega_D}(\Omega)} \cF_0(v)\,, \quad \text{ where } \cF_0(v) 
        := \cE_0(v) + \cG(v) + \wt{\cG}(v) + \cG_{\beta>d}(v)\,.
    \end{equation*}

    In addition, if $g =0$ then the same result holds if $\cG : W^{1,p}_{0,\p \Omega_D}(\Omega) \to \bbR$ is weakly continuous in the space $W^{1,p}_{0,\p \Omega_D}(\Omega)$.
\end{theorem}

\begin{theorem}\label{thm:LocLimit:Neumann}
    Assume \eqref{eq:LocalLimit:LOTAssump} and \eqref{eq:assump:rho:q}.
    For a sequence $\delta \to 0$, let $u_\delta \in \mathring{\frak{W}}^{\beta,p}[\delta;q](\Omega)$ be a function satisfying \eqref{eq:MinProb:Neumann}.
    Then $\{ u_\delta \}_\delta$ is precompact in the strong topology of $L^p(\Omega)$. Furthermore, any limit point $u$ satisfies $u \in \mathring{W}^{1,p}(\Omega)$, where $\mathring{W}^{1,p}(\Omega) := \{ v \in W^{1,p}(\Omega) \, : \, (v)_\Omega = 0 \}$, with
    \begin{equation*}
        \cF_0(u) = \min_{ v \in \mathring{W}^{1,p}(\Omega)} \cF_0(v)\,.
    \end{equation*}
\end{theorem}

\begin{theorem}\label{thm:LocLimit:Robin}
    Assume \eqref{eq:LocalLimit:LOTAssump} and \eqref{eq:assump:rho:q}.
    For a sequence $\delta \to 0$, let $u_\delta \in \frak{W}^{\beta,p}[\delta;q](\Omega)$ be a function satisfying \eqref{eq:MinProb:Robin}.
    Then $\{ u_\delta \}_\delta$ is precompact in the strong topology of $L^p(\Omega)$. Furthermore, any limit point $u$ satisfies $u \in W^{1,p}(\Omega)$ with
    \begin{equation*}
        \cF_0^R(u) = \min_{ v \in W^{1,p}(\Omega)} \cF_0^R(v)\,, 
    \end{equation*}
    where      
    \begin{equation*}
    \cF_0^R(v) := \cE_0(v) + \int_{\p \Omega} b |Tv|^p \, \rmd \sigma + \cG(v) +\wt{\cG}(v) + \cG_{\beta>d}(v)\,.
    \end{equation*}
\end{theorem}

\subsection{Examples}\label{subsec:Examples}
To demonstrate the scope of our analysis, we present several examples.

\textit{Example 1: Dirichlet Constraints.}
Let $\p \Omega_D = \p \Omega$, let $g = 0$ and let $f \in [W^{1,p}_0(\Omega)]^*$. For $m \in [1,\frac{dp}{d-p})$ if $p < d$ and any finite exponent if $p \geq d$, 
and $\Phi(t) = \frac{t^p}{p}$, our analysis shows that there exists a minimizer $u_\delta \in \frak{W}^{\beta,p}_{0,\p \Omega}[\delta;q](\Omega)$ of
\begin{equation*}
   \cE_\delta(u) + \frac{1}{m} \int_{\Omega} |K_\delta[\bar{\lambda},q,\psi]u(\bx)|^m \, \rmd \bx - \vint{f,K_\delta[\bar{\lambda},q,\psi]u}\,,
\end{equation*}
(which is unique thanks to the strict convexity of $\cE_\delta$ and the convexity of the other two terms),
and $\{u_\delta\}_\delta$ converges strongly in $L^p(\Omega)$ to a minimizer $u \in W^{1,p}_0(\Omega)$ of
\begin{equation*}
    \frac{1}{p} \int_{\Omega} |\grad u(\bx)|^p \, \rmd \bx + \frac{1}{m} \int_{\Omega} |u(\bx)|^m \, \rmd \bx - \vint{f,u}\,.
\end{equation*}
Here we have taken $\cG(u) = \frac{1}{m} \int_{\Omega} |u|^m \, \rmd \bx -\vint{f,u}$, which satisfies \eqref{eq:LowerOrderTerm}, and we have taken $\wt{\cG}_\delta = \wt{\cG}_{\beta>d} \equiv 0$.
If $f$ additionally belongs to $[\frak{W}^{\beta,p}_{0, \p \Omega}[\delta;q](\Omega)]^*$, then we could instead take $\wt{\cG}_\delta = \vint{f,\cdot}$ and $\cG=\cG_{\beta>d} \equiv 0$.
That is, the same existence result holds for the functional with the term $\vint{f,K_\delta u}$ replaced by $\vint{f,u}$. However, the local limit result does not hold, since the functional $\wt{\cG}(u) = \Vint{f,u}$ would not satisfy the condition \eqref{eq:LocalLimit:LOTAssump}. To be more precise, $\Vint{f,u}$ is not even defined for $u \in W^{1,p}_{0}(\Omega)$, which is a smaller space than $\frak{W}^{\beta,p}_{0,\p \Omega}[\delta;q](\Omega)$.

\textit{Example 2: Neumann Constraints.}
Given linear functionals $f \in [W^{1,p}(\Omega)]^*$ and $g \in [W^{1-1/p,p}(\p \Omega)]^*$, there exists a (unique) minimizer $u_\delta \in \mathring{\frak{W}}^{\beta,p}[\delta;q](\Omega)$ of
\begin{equation*}
    \cE_{\delta}(u) - \Vint{f, K_{\delta}[\bar{\lambda},q,\psi]u } - \vint{g,Tu}\,,
\end{equation*}
and $\{u_\delta\}_\delta$ converges strongly in $L^p(\Omega)$ to a minimizer $u \in \mathring{W}^{1,p}(\Omega)$ of
\begin{equation*}
    \bar{\rho} \int_{\Omega} \fint_{\bbS^{d-1}} \Phi(|\grad u(\bx) \cdot \bsomega|) \, \rmd \sigma(\bsomega) \, \rmd \bx - \vint{f , u } - \vint{g,Tu}\,.
\end{equation*}
Above, we took $\cG = -\vint{f,\cdot}$, $\wt{\cG}_\delta = \wt{\cG} = -\vint{g,T(\cdot)}$, and $\cG_{\beta>d}\equiv 0$.
Note that we do not require any compatibility condition such as $\Vint{f,1} + \Vint{g,1} = 0$, since we have not discussed any associated Euler-Lagrange equations.

\textit{Example 3: Fixed exponents and nonlinear terms.}
Let $d = 3$, $p = 2$, $\beta = d + 2s$ for some $s \in (\frac{3}{4},1)$. 
Let $f \in [W^{1,2}(\Omega)]^*$, setting $\cG = -\vint{f,\cdot}$, and let $g \in [W^{1/2,2}(\p \Omega)]^*$, setting $\wt{\cG}_\delta = \wt{\cG} = -\vint{g,T(\cdot)}$.
Let
\begin{equation*}
    \cG_{\beta>d}(u) = \int_{\Omega} u(\bx)^2 (1- u(\bx)^2) \, \rmd \bx\,,
\end{equation*}
which satisfies \eqref{eq:LowerOrderTerm2} thanks to our choice of exponents. Then our analysis shows that there exists a minimizer $u_\delta \in \frak{W}^{d+2s,2}[\delta;q](\Omega)$ of
\begin{equation*}
    \cE_{\delta}(u) + \int_{\p \Omega} b|Tu|^2 \, \rmd \sigma - \Vint{f, K_{\delta}[\bar{\lambda},q,\psi]u } - \Vint{g,Tu} + \int_{\Omega} u(\bx)^2 (1- u(\bx)^2) \, \rmd \bx\,,
\end{equation*}
and any sequence of minimizers $\{u_\delta\}_\delta$ converges to a minimizer $u \in W^{1,2}(\Omega)$ of
\begin{equation*}
    \cE_{0}(u) + \int_{\p \Omega} b|Tu|^2 \, \rmd \sigma - \Vint{f,u } - \Vint{g,Tu} + \int_{\Omega} u(\bx)^2 (1- u(\bx)^2) \, \rmd \bx\,.
\end{equation*}

Additionally, we note that our analysis also allows
an array of models satisfying a nonlocal nonlinear elliptic equation in the interior of $\Omega$.
The discussion of the strong forms of these equations will be the subject of a subsequent paper. As an illustration, we may
let $\tau>0$ be a constant, and let $q(r)$ be a $C^1$ mollification of the function $\min\{r,\frac{\tau}{2}\}$, so that $\eta(\bx) = \frac{\tau}{2}$ in  $\Omega^\tau = \{\bx \, : \, \dist(\bx, \p \Omega) > \tau \}$ and $\eta(\bx) = \dist(\bx,\p \Omega)$ otherwise.
With this choice of $\eta$, a minimizer of any of the above examples solves an Euler-Lagrange equation 
with no heterogeneous localization occurring in the interior. Different equations can be treated, with principal operator either with or without singularity on the diagonal.
If $\beta = 0$, then the operator corresponds to a $p$-Laplacian operator of convolution type.
If $\beta > d$, say $\beta = d+sp$ for some $s \in (0,1)$, then the operator corresponds to a censored $s$-fractional $p$-Laplacian.

This paper is organized as follows: the next section contains some comparability results for different nonlocal seminorms. Section \Cref{sec:buildingblockest} contains some estimates of quantities involving the heterogeneous localization that we reference throughout the paper. Properties of the boundary-localized convolution are investigated in \Cref{sec:HSEstimates}. The density of smooth functions, the trace theorem, the Poincar\'e inequalities, and the compact emebedding results are all stated and proved precisely in \Cref{sec:NonlocalSpaceFundProp}. \Cref{sec:varprob} contain the existence results for the variational problems, and the proofs of convergence to the corresponding local problems are in \Cref{sec:loclim}. 

\section{Equivalence and comparison
of nonlocal function spaces}\label{sec:PropertiesOfSpaces}
We present some results on the nonlocal function space $\frak{W}^{\beta,p}[\delta;q](\Omega)$ for $p\in [1,\infty)$ and under the assumptions
\eqref{assump:beta}, 
\eqref{assump:NonlinearLocalization}, and $\delta < \underline{\delta}_0$. Moreover,
we henceforth define $\dist(\bx,\p \Omega) = d_{\p \Omega}(\bx)$.

\subsection{Nonlocal energy spaces
for different
bulk horizon parameters}
In a spirit similar to \cite[Lemma 6.2]{tian2017trace} and \cite[Lemma 2.2]{du2022fractional}, we show the equivalence of the nonlocal function space with respect to differing values of $\delta$. This proof is representative of the types of estimates used throughout the work. It also motivates the choice of $\underline{\delta}_0 < \frac{1}{3}$ used in \eqref{eq:HorizonThreshold2}.

\begin{theorem}
\label{thm:InvariantHorizon}
    For constants $0 < \delta_1 \leq \delta_2 < \underline{\delta}_0$,
	\begin{equation*}
		\left( \frac{1-\delta_2}{2(1+\delta_2)} \right)^{\frac{d+p-\beta}{p}}[u]_{\frak{W}^{\beta,p}[\delta_2;q](\Omega)} 
    \leq [u]_{\frak{W}^{\beta,p}[\delta_1;q](\Omega)} \leq \left( \frac{\delta_2}{\delta_1} \right)^{1+(d-\beta)/p} [u]_{\frak{W}^{\beta,p}[\delta_2;q](\Omega)}
	\end{equation*}
    for all $u \in \frak{W}^{\beta,p}[\delta_2;q](\Omega)$.
\end{theorem}

\begin{proof}
	The second inequality is trivial, so the proof is devoted to the first inequality.
	Let $n \in \bbN$.
	To begin, we apply the triangle inequality to the telescoping sum for $\bx \in \Omega$ and $\bs \in B(0, \delta_2 \eta(\bx))$ where in this proof $\eta(\bx) := q(d_{\p \Omega}(\bx))$,
	\begin{equation*}
		|u(\bx+\bs) - u(\bx)| \leq \sum_{i = 1}^n \left| u \left( \bx + \frac{i}{n} \bs \right) - u \left( \bx + \frac{i-1}{n} \bs \right)\right|\,.
	\end{equation*}
Note that 
$|\bx+\frac{i}{n}\bs - \bx| \leq |\bs| < \delta_2 \eta(\bx)$, so $\bx +\frac{i}{n}\bs \in \Omega$ for $i = 0,1,\ldots,n$. Thus, setting $\bx_i := \bx + \frac{i-1}{n} \bs$ and using H\"older's inequality, we get 
	\begin{equation*}
		\begin{split}
			[u]_{ \frak{W}^{\beta,p}[\delta_2;q](\Omega) }^p \leq C_{d,p,\beta} n^{p-1} \sum_{i = 1}^n \int_{ \Omega } \int_{B(0,\delta_2 \eta(\bx) ) } \frac{ |u(\bx_i+\frac{1}{n}\bs)-u(\bx_i)|^p }{ |\bs|^{\beta} |\delta_2 \eta(\bx)|^{d+p-\beta} } \, \rmd \bs \, \rmd \bx\,.
		\end{split}
	\end{equation*}
	Now, since $\eta$ is Lipschitz with 
 a Lipschitz constant no larger than $1/3$ 
        \begin{equation*}
        \eta(\bx_i) = \eta(\bx + (i-1)\bs/n ) \leq |\bs| + \eta(\bx) \leq (\delta_2 + 1)  \eta(\bx)\,,
	\end{equation*}
	and
	\begin{equation*}
		\eta(\bx) \leq  |\bs| + \eta(\bx_i) \leq \delta_2 \eta(\bx) +  \eta(\bx_i)\,,
	\end{equation*}
 therefore, by \eqref{eq:h0property},	\begin{equation}\label{eq:VaryingLambda:SeminormComparison:Pf2}
		\frac{3}{4} \eta(\bx_i)
        \leq \frac{\eta(\bx_i)}{1+\delta_2 } \leq \eta(\bx) \leq \frac{\eta(\bx_i)}{1 - \delta_2 }
        \leq \frac{3}{2} \eta(\bx_i)
	\end{equation}
	for all $\bx \in \Omega$.
	Hence,	\begin{equation}\label{eq:VaryingLambda:SeminormComparison:Pf3}
		\begin{split}
		&[u]_{ \frak{W}^{\beta,p}[\delta_2;q](\Omega) }^p \\
        \leq& C_{d,p,\beta} n^{p-1} (1+\delta_2)^{d+p-\beta} \sum_{i = 1}^n \int_{ \Omega } \int_{B(0,\frac{\delta_2}{1-\delta_2}  \eta(\bx_i)) } \frac{ |u(\bx_i+\frac{1}{n}\bs)-u(\bx_i)|^p }{ |\bs|^{\beta} |\delta_2 \eta(\bx_i)|^{d+p-\beta} } \, \rmd \bs \, \rmd \bx\,.
		\end{split}
	\end{equation}
	Now, for $\bx \in \Omega$ and $|\bs| \leq \frac{\delta_2}{1 - \delta_2} \eta(\bx_i)$, we have
	\begin{equation*}
		\eta(\bx_i) \leq \eta(\bx) + |\bs| \leq \eta(\bx) + \frac{\delta_2}{1 - \delta_2} \eta(\bx_i)\,.
	\end{equation*}
	Therefore $ \eta(\bx_i) \leq \frac{ 1-\delta_2 }{1 - 2\delta_2 }  \eta(\bx)$ and since $\delta_2 < \underline{\delta}_0 < \frac{1}{3}$ we conclude that
	\begin{equation*}
		|\bx-\bx_i| \leq |\bs| \leq \frac{\delta_2}{1-\delta_2} \eta(\bx_i) \leq \frac{\delta_2}{1 - 2\delta_2} \eta(\bx) < \eta(\bx)\,,
	\end{equation*}
	i.e. $\bx_i \in \Omega$ for all $i = 1, \ldots, n$.
	With this we can perform a change of variables in the outer integral; letting $\by = \bx_i = \bx + \frac{i-1}{n} \bs$ in \eqref{eq:VaryingLambda:SeminormComparison:Pf3} and using \eqref{eq:VaryingLambda:SeminormComparison:Pf2}, we get
	\begin{equation*}
		\begin{split}
			&[u]_{ \frak{W}^{\beta,p}[\delta_2;q](\Omega) }^p \\
            \leq& C_{d,p,\beta} n^{p-1} (1+\delta_2)^{d+p-\beta} \sum_{i = 1}^n \int_{ \Omega } \int_{B(0,\frac{\delta_2}{1-\delta_2} \eta(\by) ) } \frac{ |u(\by+\frac{1}{n}\bs)-u(\by)|^p }{ |\bs|^{\beta} |\delta_2 \eta(\by)|^{d+p-\beta} } \, \rmd \bs \, \rmd \by \\
			=& C_{d,p,\beta} n^{p} (1+\delta_2)^{d+p-\beta} \int_{ \Omega } \int_{B(0,\frac{\delta_2}{1-\delta_2} \eta(\by) ) } \frac{ |u(\by+\frac{1}{n}\bs)-u(\by)|^p }{ |\bs|^{\beta} |\delta_2 \eta(\by)|^{d+p-\beta} } \, \rmd \bs \, \rmd \by \,.
		\end{split}
	\end{equation*}
	Now perform a change of variables in the inner integral by $\bz = \frac{\bs}{n}$ to obtain
	\begin{equation*}
		\begin{split}
			&[u]_{ \frak{W}^{\beta,p}[\delta_2;q](\Omega) }^p \\
			\leq& C_{d,p,\beta} n^{d+p-\beta} (1+\delta_2)^{d+p-\beta} \int_{ \Omega } \int_{B(0,\frac{\delta_2}{1-\delta_2} \frac{ \eta(\by) }{n}) } \frac{ |u(\by+\bz)-u(\by)|^p }{ |\bz|^{\beta} |\delta_2 \eta(\by)|^{d+p-\beta} } \, \rmd \bz \, \rmd \by \\
			=& C_{d,p,\beta} 
\left( \frac{n\delta_1 (1+\delta_2)}{\delta_2}
\right)^{d+p-\beta}
   \int_{ \Omega } \int_{
   B(0,\frac{\delta_2}{1-\delta_2} \frac{ \eta(\by)}{n}
   ) } \frac{ |u(\by+\bz)-u(\by)|^p }{ |\bz|^{\beta} | 
   \delta_1 \eta(\by)|^{d+p-\beta} } \, \rmd \bz \, \rmd \by\,.
		\end{split}
	\end{equation*}
	By taking $n \in \bbN$ such that	\begin{equation*}
		\frac{\delta_2}{\delta_1 ( 1 - \delta_2)} < n < \frac{2\delta_2}{\delta_1 ( 1 - \delta_2)} \,.
	\end{equation*}
we have
	\begin{equation*}
		\begin{split}
			&[u]_{ \frak{W}^{\beta,p}[\delta_2;q](\Omega) }^p 
			\leq \left( \frac{2(1+\delta_2)}{1-\delta_2} \right)^{d+p-\beta}  \int_{ \Omega } \int_{\Omega} \gamma_{\beta,p}[\delta_1;q](\bx,\by) |u(\bx)-u(\by)|^p  \, \rmd \by \, \rmd \bx\,,
		\end{split}
	\end{equation*}
	as desired.
\end{proof}

\subsection{Nonlocal energy spaces with varying localizations}
Let us consider a more general seminorm, which not only expands the scope of the techniques used but will also streamline the analysis of the functional $\cE_\delta$ in later sections. 
For $\rho$ satisfying \eqref{assump:VarProb:Kernel}, and $\lambda$ satisfying \eqref{assump:Localization}, define
\begin{equation*}
[u]_{\frak{V}^{\beta,p}[\delta;q;\rho,\lambda](\Omega)}^p := \int_{\Omega} \int_{\Omega} \gamma_{\beta,p}[\delta;q;\rho,\lambda](\bx,\by) |u(\by)-u(\bx)|^p \, \rmd \by \, \rmd \bx\,,
\end{equation*}
where 
\begin{equation*}
	\gamma_{\beta,p}[\delta;q;\rho,\lambda](\bx,\by) := \rho \left( \frac{|\by-\bx|}{ \delta q(\lambda(\bx)) } \right) \frac{ C_{d,\beta,p}(\rho) }{ |\bx-\by|^{\beta} } \frac{1}{ (\delta q(\lambda(\bx) ))^{d+p-\beta} }\,,
\end{equation*}
and $C_{d,\beta,p}(\rho)$ is chosen so that $C_{d,\beta,p}(\rho) \int_{\bbR^d} \frac{\rho(|\bsxi|)}{|\bsxi|^{\beta-p}} \, \rmd \bsxi = \overline{C}_{d,p}$.

We now note the independence of the nonlocal energy norm space on the specific form of the nonlocal kernel $\rho$. In particular, for a suitable range of $\delta$ we can select mollified versions of the kernel $\mathds{1}_{ \{ |\by-\bx| \leq \delta \dist(\bx,\p \Omega) \} }$ and the distance function $d_{\p \Omega}(\bx)$ to create equivalent seminorms.

\begin{theorem}
\label{thm:EnergySpaceIndepOfKernel}
Let $\rho$ be a nonnegative even function in $L^\infty(\bbR)$ with support in $(-1,1)$.
Then there exists a constant $C$ depending only on $d$, $\beta$, $p$, $\rho$, $q$, and $\kappa_0$ such that
\begin{equation*}\label{eq:EnergyComp1}
[u]_{\frak{V}^{\beta,p}[\delta;q;\rho,\lambda](\Omega)} \leq C  [u]_{\frak{W}^{\beta,p}[\delta;q](\Omega)}\,,
\;\; \forall u \in \mathfrak{W}^{\beta,p}[\delta;q](\Omega).
\end{equation*}
If in addition $\rho$ satisfies \eqref{assump:VarProb:Kernel}, then there exists a constant $c > 0$ with the same dependencies such that 
\begin{equation}\label{eq:EnergyComp2}
\;\; c [u]_{\frak{W}^{\beta,p}[\delta;q](\Omega)} \leq [u]_{\frak{V}^{\beta,p}[\delta;q;\rho,\lambda](\Omega)}\,, \;\; \forall u \in \mathfrak{W}^{\beta,p}[\delta;q](\Omega)\,.
\end{equation}
\end{theorem}

\begin{proof}
    The result is clear if \eqref{eq:EnergyComp1}-\eqref{eq:EnergyComp2} is established under the assumption \eqref{assump:VarProb:Kernel}.
    First, we have
    \begin{equation*}
        C(\rho)^{-1} \mathds{1}_{ B(0,c_\rho) }(\bx) \leq \rho(|\bx|) \leq C(\rho) \mathds{1}_{ B(0,1) }(\bx)
    \end{equation*}
    for $C(\rho) > 1$. Next, by \eqref{eq:NonlinLocAssump:Consequence} and \eqref{assump:Localization} 
    \begin{gather*}
        \frac{ q( d_{\p \Omega}(\bx)) }{C(q,\kappa_0)} \leq q(\lambda(\bx)) \leq C(q,\kappa_0)  q( d_{\p \Omega}(\bx))
    \end{gather*}
    for $C(q,\kappa_0) := C(q) >1$. Therefore
    \begin{equation*}
    \begin{split}
        &\frac{1}{C(\rho)}\int_{\Omega} \int_{ \{ |\by-\bx| < \frac{c_\rho}{C(q)} \delta q(d_{\p \Omega}(\bx))  \} } \frac{|u(\bx)-u(\by)|^p}{ |\bx-\by|^\beta |C(q) \delta q(d_{\p \Omega}(\bx))|^{d+p-\beta} } \, \rmd \by \, \rmd \bx \\
        &\leq [u]_{\frak{V}^{\beta,p}[\delta;q;\rho,\lambda](\Omega)}^p \\
        &\leq C(\rho) C(q)^{d+p-\beta} \int_{\Omega} \int_{ \{ |\by-\bx| < C(q) \delta q(d_{\p \Omega}(\bx))  \} } \frac{|u(\bx)-u(\by)|^p}{ |\bx-\by|^\beta |\delta q(d_{\p \Omega}(\bx))|^{d+p-\beta} } \, \rmd \by \, \rmd \bx \,.
    \end{split}
    \end{equation*}
   The conclusion then follows from the assumptions on $\delta$ and \Cref{thm:InvariantHorizon}.
\end{proof}

\section{Properties of heterogeneous localization functions and the associated kernels}\label{sec:buildingblockest}

We now present some properties related to the function $\eta$ and various kernels used in this work.
All the discussions are 
under the assumptions
\eqref{assump:NonlinearLocalization}, \eqref{assump:Localization},
\eqref{Assump:Kernel}, and $\delta < \underline{\delta}_0$.
Additional assumptions on $\psi$ are made for some of the results presented in 
\cref{sec:AuxMolls}.

\subsection{Spatial variations of the heterogeneous localization function}\label{sec:LocalizationFunc}
For ease of access, we record the following comparisons of 
the heterogeneous localization function $\eta_\delta$ that are frequently referred to in later discussions:

\begin{lemma}\label{lma:ComparabilityOfXandY}
    For all $\bx, \by \in \Omega$,
\begin{eqnarray}\label{eq:ComparabilityOfDistanceFxn1}
	&	(1- \kappa_1 \delta ) \eta_\delta(\bx) \leq \eta_\delta(\by) \leq (1+\kappa_1 \delta ) \eta_\delta(\bx), \; \text{ if } |\bx-\by| \leq \eta_\delta(\bx),\\
\label{eq:ComparabilityOfDistanceFxn2}
&		(1-\kappa_1 \delta ) \eta_\delta(\by) \leq \eta_\delta(\bx) \leq (1+\kappa_1 \delta ) \eta_\delta(\by), \;  \text{ if } |\bx-\by| \leq \eta_\delta(\by)\,.
	\end{eqnarray}
\end{lemma}

\begin{proof}
	It suffices to show \eqref{eq:ComparabilityOfDistanceFxn1}, since \eqref{eq:ComparabilityOfDistanceFxn2} will follow from the same arguments with the roles of $\bx$ and $\by$ interchanged.
	From the properties of $\eta_\delta$ coming from \eqref{assump:NonlinearLocalization} and \eqref{assump:Localization}, we get
	\begin{equation*}
		\eta_\delta(\by) \leq \eta_\delta(\bx) + \delta \kappa_1 |\bx-\by| \leq (1+\kappa_1 \delta ) \eta_\delta(\bx)
	\end{equation*}
	and
	\begin{equation*}
		\eta_\delta(\bx) \leq \eta_\delta(\by) + \delta  \kappa_1 |\bx-\by| \leq  \eta_\delta(\by) + \delta  \kappa_1 \eta_\delta(\bx)\,.
	\end{equation*}
\end{proof}

The next lemma is used later to facilitate a change of coordinates. To set the notation used in it, let us introduce a function $\bar{\lambda}$ that also satisfies \eqref{assump:Localization} with the same constants $\kappa_\alpha$ as those in the assumption for $\lambda$.
Likewise, we denote
$\bar{\eta}_\delta := \eta_\delta[\bar{\lambda},q]$ and
$\bar{\eta}:= \bar{\eta}_1$.
\begin{lemma}\label{lma:CoordChange2}
For a fixed $\bz \in B(0,1)$, 
    define the function $\bszeta_\bz^\veps : \Omega \to \bbR^d$ by 
    \begin{equation*}
    \bszeta_\bz^\veps(\bx) := \bx + \bar{\eta}_\veps(\bx) \bz\,, \quad
    \forall \veps \in (0,\underline{\delta}_0)\,.
    \end{equation*}
    Then there exists a constant $\bar{c} = \bar{c}(q,\kappa_0,\kappa_1) \geq \kappa_1$
    such that for all $\bx$, $\by \in \Omega$, for all $\delta \in (0,\underline{\delta}_0)$, and for all $\veps \in (0,\veps_0)$, where  $\veps_0 := \frac{1}{3} \min \{ 1, \frac{1}{\bar{c}} \} \in (0,\underline{\delta}_0)$,
    we have 
    \begin{equation}\label{eq:bszetaproperties}
        \begin{gathered}
            \det \grad \bszeta_\bz^\veps(\bx) = 1 + \grad \bar{\eta}_\veps(\bx) \cdot \bz > 1-\kappa_1 \veps > \frac{2}{3}\,, \\
            \bszeta_\bz^\veps(\bx) \in \Omega \text{ and } 0 < (1- \bar{c} \veps) \eta(\bx) \leq \eta(\bszeta_\bz^\veps(\bx)) \leq (1+\bar{c}\veps) \eta(\bx) \,, \\
            0 < (1- \bar{c} \veps) |\bx-\by| \leq | \bszeta_\bz^\veps(\bx) - \bszeta_\bz^\veps(\by)| \leq (1+\bar{c} \veps) |\bx-\by|\,, \text{ and } \\
            |\bx-\by| \leq \eta_\delta(\bx) \quad \Rightarrow \quad |\bszeta_{\bz}^{\veps}(\bx) - \bszeta_{\bz}^{\veps}(\by)| \leq \frac{1+ \bar{c} \veps}{1-\bar{c} \veps} \eta_\delta( \bszeta_{\bz}^{\veps}(\bx) )\,.
        \end{gathered}
    \end{equation} 
    Moreover, if $\bar{\eta} 
    \equiv \eta$ then $\bar{c} = \kappa_1$ and $\veps_0 = \underline{\delta}_0$ can be chosen.
\end{lemma}

\begin{proof}
    The positive lower bound on $\det \grad \bszeta_\bz^\veps$ follows from the properties of $q$ and $\bar{\lambda}$ and the assumption on $\underline{\delta}_0$. 
    Now, by \eqref{assump:Localization} and \eqref{eq:NonlinLocAssump:Consequence}, we have
    \begin{equation*}
        \frac{\bar{\lambda}(\bx)}{\kappa_0^2} \leq \lambda(\bx) \leq \kappa_0^2 \bar{\lambda}(\bx)\,, \qquad \text{ for } \bx \in \Omega\,,
            \end{equation*}
            which implies that
     \begin{equation*}       
    \frac{\eta(\bx)}{ C_q \kappa_0^{2 \log_2(C_q)} } \leq \bar{\eta}(\bx) \leq C_q \kappa_0^{2 \log_2(C_q)} \eta(\bx)\,,
    \end{equation*}
    and so the second line in \eqref{eq:bszetaproperties} follows from the Lipschitz continuity of $\eta$:
    \begin{equation*}
    \begin{gathered}
        \eta(\bszeta_\bz^\veps(\bx)) \geq \eta(\bx) - \kappa_1 \bar{\eta}_\veps(\bx) |\bz| \geq  (1- \kappa_1 C_q \kappa_0^{2 \log_2(C_q)} \veps ) \eta(\bx)
        =: (1-\bar{c} \veps) \eta(\bx), \text{ and} \\
        \eta(\bszeta_\bz^\veps(\bx)) \leq \eta(\bx) + \kappa_1 \bar{\eta}_\veps(\bx) |\bz| \leq (1+\bar{c}\veps) \eta(\bx)\,.
    \end{gathered}
    \end{equation*}
    The third line follows from
    \begin{equation*}
        \big| | \bszeta_\bz^\veps(\bx) - \bszeta_\bz^\veps(\by)| - |\bx-\by| \big| \leq |\eta_\veps(\bx) - \eta_\veps(\by)| |\bz| \leq \kappa_1 \veps |\bx-\by|\,,
    \end{equation*}
    and the fourth line of \eqref{eq:bszetaproperties} follows from the second and third lines.    
\end{proof}

\subsection{Mollifier kernels}\label{sec:AuxMolls}
For any  function $\psi: [0,\infty) \to \bbR$, we
define
\begin{equation}\label{eq:OperatorKernelDef}
	\begin{split}
		\psi_{\delta}[\lambda,q](\bx,\by) &:= \frac{1}{\eta_\delta[\lambda,q](\bx)^{d}} {\psi} \left( \frac{|\by-\bx|}{\eta_\delta[\lambda,q](\bx)} \right) \,.
	\end{split}
\end{equation}
In particular, $\psi_{\delta}[\lambda,q]$
defines a \textit{boundary-localizing} mollifier
corresponding to a standard mollifier $\psi$ described in 
\eqref{Assump:Kernel}. 
We
write $\psi_{\delta}[\lambda,q]$ as simply $\psi_{\delta}$ whenever the context is clear.
Note that $\int_{\Omega} \psi_\delta(\bx,\by) \, \rmd \by = 1$ for all $\bx \in \Omega$ and for all $\delta < \underline{\delta}_0$. This is not the case when the arguments are reversed, and so we define the function
\begin{equation}\label{eq:KernelIntegral}
    \Psi_{\delta}(\bx):= \int_{\Omega} \psi_{\delta}(\by,\bx) \, \rmd \by\,.
\end{equation}
Let us investigate the properties of $\Psi_{\delta}$
below.

\begin{lemma}\label{lma:KernelIntegral:SqDist}
    Let $\psi$ be a nonnegative even function in $C^0(\bbR)$ with support in $[-1,1]$.
	Then there exists a constant $C$ depending only on $d$, $\psi$, $\lambda$, $q$, and $\kappa_1$ such that 
\begin{equation}\label{eq:KernelIntFunction:UpperBound}
		\Psi_{\delta}(\bx) \leq C\,
  ,\quad \forall \bx \in \Omega.
	\end{equation}
\end{lemma}

\begin{proof}
	Since $\supp \psi \subset [-1,1]$ we obtain the upper bound from \eqref{eq:ComparabilityOfDistanceFxn1}:
	\begin{equation*}
		\begin{split}
		\int_{\Omega} \frac{1}{\eta_\delta(\by)^d} \psi \left( \frac{|\bx-\by|}{\eta_\delta(\by) } \right) \, \rmd \by 
		&\leq \| \psi \|_{L^{\infty}([0,1])} \int_{ \{ |\by-\bx| \leq \lambda_{ \delta}(\by) \} } \frac{1}{\eta_\delta(\by)^{d}} \, \rmd \by \\
		&\leq C(\psi) \int_{ \{ |\by-\bx| \leq (1+\kappa_1 \delta)\lambda_{ \delta}(\bx) \} } \frac{1}{(1-\kappa_1 \delta )^d \eta_\delta(\bx)^{d}} \, \rmd \by \\
		&= C\,.
		\end{split}
	\end{equation*}
\end{proof}

We now turn to the derivatives of $\psi_\delta$.
It is clear that $\psi_\delta(\bx,\by) \in C^k(\Omega \times \Omega)$ whenever \eqref{assump:Localization}, \eqref{assump:NonlinearLocalization}, and \eqref{Assump:Kernel} are satisfied for the same $k \in \bbN \cup \{\infty\}$. We record several estimates on the derivative of the kernel that we will use.

\begin{theorem}\label{thm:gradientpsi}
Let $\psi$  satisfy \eqref{Assump:Kernel} for some $k \geq 1$.
Then there exists $C = C(d,\psi)$ such that
    \begin{equation}\label{eq:KernelDerivativeEstimate}
		|\grad _{\bx} \psi_{\delta}(\bx,\by)| \leq \frac{C}{\eta_\delta(\bx) } \Big( \psi_{\delta}(\bx,\by) 
		+ (|\psi'|)_{\delta}(\bx,\by) \Big)\,,  \quad \forall \bx, \by \in \Omega.
    \end{equation}
\end{theorem}

\begin{proof} 
	This follows by direct computation and the properties of $\psi$:
    \begin{equation}\label{eq:MollifierDerivativeFormula}
		\begin{split}
        \grad_{\bx} {\psi}_{\delta}(\bx,\by) &= (\psi')_\delta(\bx,\by) \left( \frac{\bx-\by}{|\bx-\by|} \frac{1}{\eta_\delta(\bx)} - \frac{|\bx-\by|}{\eta_\delta(\bx)} \frac{\grad \eta_\delta(\bx)}{\eta_\delta(\bx)} \right) \\
		& \quad - \psi_{\delta}(\bx,\by)  \frac{ \grad \eta_\delta(\bx) }{\eta_\delta(\bx)} d \,.
		\end{split}
	\end{equation}
	Thus, using the support of $\psi$ and that $|\grad \eta_\delta| \leq 1/3$, we see the result.
\end{proof}

\begin{corollary}
    Let $\psi$ satisfy \eqref{Assump:Kernel} for some $k \geq 1$, and let $\alpha \in \bbR$. Then there exists $C = C(d,\psi,\kappa_1,\alpha)$ such that
\begin{equation}\label{eq:KernelIntegralDerivativeEstimate}
		\int_{\Omega} |\eta_\delta(\by)|^{\alpha} |\grad_{\bx} \psi_{\delta}(\bx,\by)| \, \rmd \by
		\leq \frac{C}{\eta_\delta(\bx)^{1-\alpha}}\,, \quad \forall \, \bx \in \Omega.
	\end{equation}
\end{corollary}

\begin{proof}
    We first apply  \eqref{eq:KernelDerivativeEstimate} and then 
    use \eqref{eq:ComparabilityOfDistanceFxn1}.
\end{proof}

\section{Properties of Boundary-localized convolutions}\label{sec:HSEstimates}
Our discussion in his section, unless indicated otherwise, is again
under the assumptions \eqref{assump:beta}, 
\eqref{assump:NonlinearLocalization}, \eqref{assump:Localization},
\eqref{Assump:Kernel}, and $\delta < \underline{\delta}_0$.

\subsection{General properties and consistency on the boundary}

We present the following theorem and lemma without proof, as it is straightforward to verify.

\begin{theorem}\label{thm:ConvProp:UniformCase}
	Let $u \in C^0(\overline{\Omega})$. Then $K_{\delta} u \in  C^0(\overline{\Omega})$. Moreover, $K_{\delta} u(\bx) = u (\bx)$ for all $\bx \in \p \Omega$, and $K_{\delta} u \to u$ uniformly on $\overline{\Omega}$ as $\delta \to 0$.
\end{theorem}

\begin{lemma}\label{lma:SmoothnessOfConv}
    Let $\psi$ satisfy \eqref{Assump:Kernel} for some $k=k_\psi \geq 0$, let $\lambda$ satisfy \eqref{assump:Localization} for some $k=k_\lambda \geq 0$, and let $q$ satisfy \eqref{assump:NonlinearLocalization} for some $k=k_q \geq 1$. 
    Then, for any $u \in L^1_{loc}(\Omega)$, $K_{\delta} u \in C^{k}(\Omega)$ where $k = \min \{ k_q, k_\lambda, k_\psi \}$.
\end{lemma}

For $t > 0$, we define the sets
\begin{equation*}
    \Omega_{t;\lambda,q} := \{ \bx \in \Omega \, : \, q(\lambda(\bx)) < t \} \,, \quad \text{ and } \quad \Omega^{t;\lambda,q} := \{ \bx \in \Omega \, : \, q(\lambda(\bx)) \geq t \}\,.
\end{equation*}

\begin{lemma}\label{lma:SupportOfConv}
    Suppose that $\varphi \in C^0(\overline{\Omega})$ has compact support in $\Omega$, i.e. there exists $c_\varphi > 0$ such that $\supp \varphi \subset \Omega^{c_\varphi;\lambda,q}$.
    Then $K_{\delta} \varphi$ has compact support with $\supp K_{\delta} \varphi \subset \Omega^{ \frac{ 1 }{1+ \kappa_1 \delta} c_{\varphi}; \lambda,q}$.
\end{lemma}

\begin{proof}
    By \eqref{eq:ComparabilityOfDistanceFxn1} and \eqref{eq:ComparabilityOfDistanceFxn2}, whenever $\eta(\bx) <\frac{ 1 }{1+\kappa_1 \delta } c_\varphi $
	we have
	\begin{equation*}
		\{ \by \, : |\bx-\by| \leq \eta_\delta(\bx) \} \subset \Omega_{c_\varphi;\lambda,q} \,.
	\end{equation*}
	Therefore, since the domain of integration in $K_{\delta}\varphi(\bx)$ and $\supp \varphi$ are disjoint, we have $K_{\delta} \varphi(\bx) = 0$. 
\end{proof}

\subsection{Classical function space estimates}
We first show some estimates analogous to those in
\cref{thm:convolution-estimate} for functions in classical Lebesgue and Sobolev spaces.

\begin{theorem}\label{thm:ConvEst1}
     Let $1 \leq p \leq \infty$. There exists a constant $C_0 > 0$ depending only on $d$, $p$, $\psi$, and $\kappa_1$ such that
	\begin{equation}\label{eq:ConvEst:Lp}
		\Vnorm{K_{\delta} u}_{L^p(\Omega)} \leq C_0 \Vnorm{u}_{L^p(\Omega)}, \qquad \forall u \in L^p(\Omega)\,.
	\end{equation}
	If additionally \eqref{Assump:Kernel} is satisfied for $k=k_\psi \geq 1$, then there exists a constant $C_1 > 0$ depending only on $d$, $p$, $\psi$, and $\kappa_1$ such that
	\begin{equation}\label{eq:ConvEst:W1p}
		\Vnorm{\grad K_{\delta} u}_{L^p(\Omega)} \leq C_1 \Vnorm{u}_{W^{1,p}(\Omega)}\,.
	\end{equation}
\end{theorem}

\begin{proof}
	First we prove \eqref{eq:ConvEst:Lp} for $1 \leq p < \infty$. 
	By H\"older's inequality, Tonelli's theorem, and \eqref{eq:KernelIntFunction:UpperBound},
	\begin{equation*}
		\Vnorm{ K_\delta u }_{L^p(\Omega)}^p \leq \int_{\Omega} \left( \int_{\Omega} \psi_\delta(\bx,\bz) \, \rmd \bz \right)^{p-1} \int_{\Omega} \psi_\delta(\bx,\by) |u(\by)|^p \, \rmd \by \, \rmd \bx \leq C \Vnorm{u}_{L^p(\Omega)}^p\,.
	\end{equation*}
	The inequality when $p = \infty$ is trivial.
	
	To prove \eqref{eq:ConvEst:W1p} it suffices to show that
	\begin{equation}\label{eq:ConvEst1:W1p:Pf1}
		\begin{split}
			\grad K_{\delta}u(\bx) = \int_{\Omega} \psi_{\delta}(\bx,\by) \left[ \bI - \frac{(\bx-\by) \otimes \grad \eta_\delta(\bx)}{\eta_\delta(\bx)} \right] \grad u(\by) \, \rmd \by\,,
		\end{split}
	\end{equation}
	since then
	\begin{equation*}
		|\grad K_{\delta} u(\bx)| \leq (1+\kappa_1) K_\delta [|\grad u|](\bx)\,,
	\end{equation*}
	from which \eqref{eq:ConvEst:W1p} follows by applying \eqref{eq:ConvEst:Lp}.
	First,
	\begin{equation*}
		\begin{split}
			\grad_{\by} \psi_{\delta}(\bx,\by) 
			&= \psi' \left( \frac{|\by-\bx|}{\eta_\delta(\bx)} \right) \frac{\by-\bx}{|\by-\bx|} \frac{1}{\eta_\delta(\bx)^{d+1}}\,,
		\end{split}
	\end{equation*}
	and so from the formula \eqref{eq:MollifierDerivativeFormula} we see that
	\begin{equation*}
		\begin{split}
	&		\grad_{\bx} \psi_{\delta}(\bx,\by) 
			= - \grad_{\by} \left[ \psi_{\delta} \left( \bx,\by \right) \right] \left( \bI - \frac{(\bx-\by) \otimes \grad \eta_\delta(\bx) }{\eta_\delta(\bx)} \right)
   - d \frac{\grad \eta_\delta(\bx)}{\eta_\delta(\bx)} \psi_\delta(\bx,\by) \\
   &\quad =
    - \grad_{\by} \left[ \psi_{\delta} \left( \bx,\by \right) \right] \left( \bI - \frac{(\bx-\by) \otimes \grad \eta_\delta(\bx) }{\eta_\delta(\bx)} \right)
+ \div_{\by}[\bx-\by]
\frac{\grad \eta_\delta(\bx)}{\eta_\delta(\bx)} \psi_\delta(\bx,\by) \\
   &\quad =
    -  \div_{\by}
 \left[ \psi_{\delta} \left( \bx,\by \right) \left( \bI - \frac{(\bx-\by) \otimes \grad \eta_\delta(\bx) }{\eta_\delta(\bx)} \right) \right]   
   \,.
		\end{split}
	\end{equation*}
Thus,
 	\begin{equation*}
		\begin{split}
			\grad K_{\delta} u(\bx)
			&= \int_{\Omega} \grad_\bx \psi_{\delta}(\bx,\by) u(\by) \, \rmd \by \\
			&= -  \int_{\Omega} 
    \div_{\by}
 \left[ \psi_{\delta} \left( \bx,\by \right) \left( \bI - \frac{(\bx-\by) \otimes \grad \eta_\delta(\bx) }{\eta_\delta(\bx)} \right) \right] \, \rmd \by  
   \,. 
   		\end{split}
	\end{equation*}
	Now, for any fixed $\bx \in \Omega$, the set $\{ \by \, : \, |\by-\bx| < \eta_\delta(\bx) \} \Subset \Omega$, so the functions $\psi_{\delta}(\bx,\by)$ and $\left( \bI - \frac{(\bx-\by) \otimes \grad \eta_\delta(\bx) }{\eta_\delta(\bx)} \right) u(\by)$ are integrable over $\by \in \Omega$, as are their gradients. Moreover, their product vanishes for $\by \in \p \Omega$. Therefore no boundary term appears when applying the divergence theorem, which leads exactly to  \eqref{eq:ConvEst1:W1p:Pf1}.
\end{proof}

\begin{theorem}\label{thm:TraceOfConv}
Assume \eqref{Assump:Kernel} for $k =k_\psi\geq 1$. Let $1 < p < \infty$ and denote the trace operator $T : W^{1,p}(\Omega) \to W^{1-1/p,p}(\p \Omega)$. Then $T K_{\delta} u = T u$ for all $u \in W^{1,p}(\Omega)$.
\end{theorem}

\begin{proof}
For  $\wt{u} \in C^{\infty}(\overline{\Omega})$, we have $K_\delta \wt{u} = \wt{u}$ on $\p \Omega$ by \Cref{thm:ConvProp:UniformCase}. The conclusion then follows from the density of $C^{\infty}(\overline{\Omega})$ in $W^{1,p}(\Omega)$
and \Cref{thm:ConvEst1}.
\end{proof}

\subsection{Nonlocal function space estimates}\label{sec:NonlocalFxnSpEst}

\begin{theorem}\label{thm:diffuKu}
   For $1 \leq p < \infty$ and $u \in L^p(\Omega)$, we have
\begin{equation}\label{eq:Kdeltaerror1}
		\Vnorm{u - K_{\delta} u }_{L^p(\Omega)}^p \leq \iintdm{\Omega}{\Omega}{ \psi_{\delta}(\bx,\by) |u(\bx)-u(\by)|^p }{\by}{\bx}\,.
	\end{equation}
 Consequently, \eqref{eq:KdeltaError} holds.
\end{theorem}

\begin{proof}
	H\"older's inequality gives
	\begin{equation*}
		\begin{split}
	&		\Vnorm{u - K_{\delta} u}_{L^p(\Omega)}^p = \intdm{\Omega}{  \left( \intdm{\Omega}{ \psi_{\delta}(\bx,\by) (u(\by)-u(\bx)) }{\by} \right)^p }{\bx} \\
			&\quad\leq \intdm{\Omega}{  \left( \intdm{\Omega}{\psi_{\delta}(\bx,\by) }{\by} \right)^{p-1} \left( \intdm{\Omega}{ \psi_{\delta}(\bx,\by) |u(\by)-u(\bx)|^p }{\by} \right) }{\bx} \\
			&\quad = \iintdm{\Omega}{\Omega}{ \psi_{\delta}(\bx,\by) |u(\bx)-u(\by)|^p }{\by}{\bx}\,,
		\end{split}
	\end{equation*}
 which is \eqref{eq:Kdeltaerror1}.
Now, we have the estimate
\begin{equation*}
		\begin{gathered}
		\psi_{\delta}(\bx,\by) \leq \frac{C(q,\kappa_0) \delta^p \diam(\Omega)^{p}}{\eta_\delta(\bx)^{p}} \psi_{\delta}(\bx,\by) \leq C \frac{\delta^p \diam(\Omega)^{p} }{C_{d,\beta,p}(\psi)} \gamma_{\beta,p}[\delta;q;\psi,\lambda](\bx,\by)
		\end{gathered}
	\end{equation*}
in the right-hand side integral of \eqref{eq:Kdeltaerror1}, thanks to the support of $\psi$. Then \eqref{eq:KdeltaError} follows from \Cref{thm:EnergySpaceIndepOfKernel}.
\end{proof}

\begin{theorem}\label{thm:Convolution:DerivativeEstimate}
    Assume \eqref{Assump:Kernel} for $k=k_\psi \geq 1$ and let $1 \leq p < \infty$. 
    Then there exists $C > 0$ depending only on $d$, $p$, $\psi$, $q$, $\kappa_0$ and $\kappa_1$ such that for all $u \in L^p(\Omega)$,
	\begin{equation}\label{eq:ConvEst:Deriv}
		\begin{split}
			\Vnorm{ \grad K_{\delta} u }_{L^p(\Omega)}^p
			\leq C \iintdm{\Omega}{\Omega}{ 
   \frac{(\psi_{\delta}(\bx,\by)
   +(|\psi'|)_{\delta}(\bx,\by) )}{ \eta_\delta(\bx)^{p} }
   |u(\bx)-u(\by)|^p }{\by}{\bx}\,.
		\end{split}
	\end{equation}
\end{theorem}

\begin{proof}
	Assume the right-hand side of \eqref{eq:ConvEst:Deriv} is finite. Since $\int_{\Omega} \psi_\delta(\bx,\by) \, \rmd \by = 1$, its gradient in $\bx$ vanishes, and so
	\begin{equation*}
		\begin{split}
			\grad K_{\delta}u (\bx) &= \intdm{\Omega}{\grad_{\bx} \psi_{\delta}(\bx,\by) u(\by)}{\by}  
			= \intdm{\Omega}{\grad_{\bx} \psi_{\delta}(\bx,\by) (u(\by)-u(\bx))}{\by}\,.
		\end{split}
	\end{equation*}
    Therefore, by H\"older's inequality 
	\begin{equation*}
		\begin{split}
			&\Vnorm{ \grad K_{\delta} u }_{L^p(\Omega)}^p \\
			&\leq \int_{\Omega}  \left( \eta_\delta(\bx)  \intdm{\Omega}{ |\grad_{\bx} \psi_{\delta}(\bx,\bz)|}{\bz}
   \right)^{p-1}
   \intdm{\Omega}{\frac{|\grad_{\bx} \psi_{\delta}(\bx,\by)|}{\eta_\delta(\bx)^{p-1}} |u(\by)-u(\bx)|^p}{\by} \rmd \bx\,.
		\end{split}
	\end{equation*}
	We use \eqref{eq:KernelIntegralDerivativeEstimate} 
	to get
	\begin{equation}\label{eq:ConvEst:Deriv:Pf1}
		\begin{split}
		\Vnorm{\grad K_{\delta} u }_{L^p(\Omega)}^p &\leq C \int_{\Omega} \intdm{\Omega}{\frac{|\grad_{\bx} \psi_{\delta}(\bx,\by)|}{\eta_\delta(\bx)^{p-1} } |u(\by)-u(\bx)|^p}{\by}  \, \rmd \bx\,.
			\end{split}
	\end{equation}
	From \eqref{eq:KernelDerivativeEstimate} we have
	\begin{equation*}
		\frac{|\grad_{\bx} \psi_{\delta}(\bx,\by)|}{\eta_\delta(\bx)^{p-1}  } 
        \leq C \frac{  \psi_{\delta}(\bx,\by) +  (|\psi'|)_{\delta}(\bx,\by) }{ \eta_\delta(\bx)^p }\,.
	\end{equation*}
	Using the previous estimate in \eqref{eq:ConvEst:Deriv:Pf1} gives \eqref{eq:ConvEst:Deriv}.
\end{proof}

As a consequence of the characterization of nonlocal spaces in \Cref{thm:EnergySpaceIndepOfKernel}, we obtain the following corollary that leads to \eqref{eq:Intro:ConvEst:Deriv}:
\begin{corollary}\label{cor:RegularityOfHSOp}
    Let $p \in [1,\infty)$ and assume \eqref{Assump:Kernel} for $k=k_\psi \geq 1$.
    Then \eqref{eq:Intro:ConvEst:Deriv} holds, i.e. there exists $C >0$ depending only on $d$, $\beta$, $p$, $\psi$, $q$, $\kappa_0$ and $\kappa_1$ such that
\begin{equation}\label{eq:ConvEst:Deriv:Cor}
		\begin{split}
			\Vnorm{ \grad K_{\delta} u }_{L^p(\Omega)} 
			\leq C [u]_{\frak{W}^{\beta,p}[\delta;q](\Omega)}\,,
   \qquad \forall u \in \frak{W}^{\beta,p}[\delta;q](\Omega).
		\end{split}
	\end{equation}
\end{corollary}

The above corollary and the \cref{thm:Convolution:DerivativeEstimate} together gives the result of \cref{thm:convolution-estimate}.

Additionally, the presence of the singular term $|\bx-\by|^{-\beta}$ allows us to obtain precise continuity estimates of $K_\delta$ in fractional Sobolev spaces when $\beta > d$. The proof additionally motivatives the choice of threshold in \eqref{eq:HorizonThreshold2}.

\begin{theorem}\label{thm:diffuKu:Fractional}
    Suppose additionally that $\beta > d$. Then there exists a constant $C>0$ depending only on $d$, $\beta$, $p$, $\psi$, $q$, $\kappa_0$, and $\kappa_1$ such that for all $\delta$ satisfying \eqref{eq:HorizonThreshold2}
    \begin{equation*}
        [u - K_\delta u]_{W^{ (\beta-d)/p,p }(\Omega)} \leq C (\delta q(\diam(\Omega)) )^{ \frac{d+p-\beta}{p} } [u]_{ \frak{W}^{\beta,p}[\delta;q](\Omega) }\,.
    \end{equation*}
\end{theorem}

\begin{proof}
We will estimate the left-hand side by four separate terms, each of which will be bounded by the right-hand side up to a constant. By Jensen's inequality
    \begin{equation*}
        \begin{split}
            &[u-K_\delta u]_{W^{\frac{\beta-d}{p},p}(\Omega)}^p \\
            &\leq C_{d,\beta,p} \int_{B(0,1)} \psi(|\bz|) \int_{\Omega} \int_{\Omega} \frac{ |u(\bszeta_\bz^\delta(\bx)) - u(\bszeta_\bz^\delta(\by)) - u(\bx) +u(\by)|^p }{|\bx-\by|^{\beta} }  \, \rmd \by \, \rmd \bx \, \rmd \bz\,.
        \end{split}
    \end{equation*}
    Split the right-hand side integral as $I+II$, where $I$ is the same integrand over the domain $(B(0,1) \times \Omega \times \Omega) \cap \{ |\bx-\by| \geq \max \{ \eta_\delta(\bx), \eta_\delta(\by) \} |\bz| \} $, and $II$ is the integrand over the domain $(B(0,1) \times \Omega \times \Omega ) \cap \{ |\bx-\by| < \max \{ \eta_\delta(\bx), \eta_\delta(\by) \} |\bz| \}$.
    
    We estimate $I \leq C(p,\beta,p,\psi) (I_1 + I_2)$, where
    \begin{equation*}
        I_1 := \int_{B(0,1)} \int_{\Omega} \int_{\Omega \cap \{ |\bx-\by| \geq \eta_\delta(\bx)|\bz| \} } \frac{|u(\bszeta_\bz^\delta(\bx)) - u(\bx)|^p }{|\bx-\by|^{\beta}} \, \rmd \by \, \rmd \bx \, \rmd \bz\,, 
    \end{equation*}
    and $I_2$ is defined similarly, with the roles of $\bx$ and $\by$ exchanged. 
    The $\by$-integral in $I_1$ is bounded from above by $C(d,\beta) (\eta_\delta(\bx)|\bz|)^{\beta-d}$, and so
    \begin{equation*}
        I_1 \leq C(d,\beta) \int_{\Omega} \int_{B(0,1)} |\eta_\delta(\bx) \bz|^d \frac{ |u(\bx + \eta_\delta(\bx) \bz) - u(\bx)|^p }{ |\eta_\delta(\bx) \bz|^{\beta} } \, \rmd \bz \, \rmd \bx\,.
    \end{equation*}
    Letting $\by = \bx + \eta_\delta(\bx)\bz$ and then using \Cref{thm:EnergySpaceIndepOfKernel}, we see that 
    $$
    I_1 \leq C(d,\beta,p,q) q(\diam(\Omega))^{d+p-\beta} \delta^{d+p-\beta} [u]_{\frak{W}^{\beta,p}[\delta;q](\Omega)}^p\,,
    $$
    and a similar estimate holds for $I_2$.

    Now, using \Cref{lma:ComparabilityOfXandY} to enlarge the region of integration, we estimate
    \begin{equation*}
        \begin{split}
        II &\leq 2^{p-1} \int_{B(0,1)} \int_{\Omega} \int_{\Omega \cap \{ |\bx-\by| \leq \frac{1}{1-\kappa_1 \delta} \eta_\delta(\bx) \} } \frac{ |u(\bszeta_\bz^\delta(\bx)) - u(\bszeta_\bz^\delta(\by))|^p }{|\bx-\by|^{\beta} } \, \rmd \by \, \rmd \bx \, \rmd \bz \\
            &\quad + 2^{p-1}  \int_{B(0,1)} \int_{\Omega} \int_{\Omega \cap \{ |\bx-\by| \leq \frac{1}{1- \kappa_1 \delta} \eta_\delta(\bx) \} } \frac{ |u(\bx) -u(\by)|^p }{|\bx-\by|^{\beta} } \, \rmd \by \, \rmd \bx \, \rmd \bz =: II_1 + II_2\,.
        \end{split}
    \end{equation*}
    By the identities in \Cref{lma:CoordChange2}, 
    \begin{equation*}
    \begin{split}
        II_1 \leq 
        \frac{2^{p-1}}{ (1+ \kappa_1 \delta)^\beta } &\int_{B(0,1)} \int_{\Omega} \int_{\Omega \cap \{ |\bszeta_\bz^\delta(\bx)-\bszeta_\bz^\delta(\by)| \leq \frac{1+ \kappa_1 \delta}{(1- \kappa_1 \delta)^2} \eta_\delta(\bszeta_\bz^\delta(\bx)) \} } \\
            &\qquad \qquad \frac{ |u(\bszeta_\bz^\delta(\bx)) - u(\bszeta_\bz^\delta(\by))|^p }{|\bszeta_\bz^\delta(\bx)-\bszeta_\bz^\delta(\by)|^{\beta} } \, \rmd \by \, \rmd \bx \, \rmd \bz\,.
    \end{split}
    \end{equation*}
    Then we apply the change of variables $\bar{\by} = \bszeta_\bz^\delta(\by)$, $\bar{\bx} = \bszeta_\bz^\delta(\bx)$. Since $\delta < \bar{\delta}_0$, this permits us to apply \Cref{thm:InvariantHorizon} and obtain
    \begin{equation*}
    \begin{split}
        II_1 &\leq C(p,\beta,\kappa_1) \int_{\Omega} \int_{\Omega \cap \{|\bar{\bx}-\bar{\by}|\leq \eta_\delta(\bar{\bx}) \} } \frac{ |u(\bar{\bx}) - u(\bar{\by})|^p }{ |\bar{\bx}-\bar{\by}|^\beta } \, \rmd \bar{\by} \, \rmd \bar{\bx} \\
        &\leq C(p,\beta,d,q,\kappa_0,\kappa_1) q(\diam(\Omega))^{d+p-\beta} \delta^{d+p-\beta} [u]_{\frak{W}^{\beta,p}[\delta;q](\Omega)}^p\,,
    \end{split}
    \end{equation*}
    where we additionally used \Cref{thm:EnergySpaceIndepOfKernel}. Finally, $II_2$ can be estimated in a similar but more straightforward way using \Cref{thm:InvariantHorizon} and \Cref{thm:EnergySpaceIndepOfKernel}.
\end{proof}

\subsection{Convergence in the local limit}\label{sec:Localization}
\begin{theorem}\label{thm:ConvergenceOfConv}
    Let $1 \leq p < \infty$, and let
    $u \in L^p(\Omega)$. Then
\begin{equation}\label{eq:ConvergenceOfConv}
		\lim\limits_{\delta \to 0} \Vnorm{K_{\delta} u - u}_{L^p(\Omega)} = 0\,.
\end{equation}
\end{theorem}

\begin{proof}
    The result follows from the density of $C^\infty(\overline{\Omega})$ in $L^p(\Omega)$, from the $L^p(\Omega)$-continuity of $K_\delta$ contained in \eqref{eq:ConvEst:Lp}, and from the uniform convergence of $K_\delta u$ to $u$ in \Cref{thm:ConvProp:UniformCase}.
\end{proof}

\begin{theorem}\label{thm:ConvergenceOfConv:W1p}
    Assume \eqref{Assump:Kernel} is satisfied for $k \geq 1$. Let $1 \leq p < \infty$, and let $u \in W^{1,p}(\Omega)$. Then
\begin{equation}\label{eq:ConvergenceOfConv:W1p}
		\lim\limits_{\delta \to 0} \Vnorm{K_{\delta} u - u}_{W^{1,p}(\Omega)} = 0\,.
	\end{equation}
\end{theorem}

\begin{proof}
    Note that
    \begin{equation*}
        \int_{\Omega} \psi_{\delta}(\bx,\by) \left[ \bI - \frac{(\bx-\by) \otimes \grad \eta_\delta(\bx)}{\eta_\delta(\bx)} \right] \, \rmd \by = \bI \text{ for all } \bx \in \Omega\,.
    \end{equation*}
    Therefore for any $v \in C^{\infty}(\overline{\Omega})$ we have by \eqref{eq:ConvEst1:W1p:Pf1}
    \begin{equation*}
    \begin{split}
        | \grad K_{\delta}v(\bx) - \grad v(\bx) |
        &= \left| \int_{\Omega} \psi_{\delta}(\bx,\by) \left[ \bI - \frac{(\bx-\by) \otimes \grad \eta_\delta(\bx)}{\eta_\delta(\bx)} \right] (\grad v(\by) - \grad v(\bx)) \, \rmd \by \right| \\
        &\leq C(\psi,\kappa_1) \sup_{ B(\bx,\eta_\delta(\bx)) } |\grad v(\bz) - \grad v(\bx)| \to 0 \text{ as } \delta \to 0\,.
    \end{split}
    \end{equation*}
    The limit is independent of $\bx$ since $v$ is uniformly continuous on $\Omega$. Thus $K_\delta v \to v$ uniformly on $\overline{\Omega}$. From here the proof is the same as that of \Cref{thm:ConvergenceOfConv}.
\end{proof}

\section{The nonlocal function space: fundamental properties}\label{sec:NonlocalSpaceFundProp}

In this section, we present a few important properties on the nonlocal energy spaces such as the density of smooth functions, the trace theorems, Poincar\'e inequalities, and $L^p$-compactness theorems. 
All of these results are important ingredients in the later proofs of the well-posedness of nonlocal problems with local boundary conditions.

\subsection{Density of Smooth functions}\label{subsec:Density}
We continue the discussion in this subsection under the assumptions \eqref{assump:beta}, \eqref{assump:NonlinearLocalization}, $\delta < \underline{\delta}_0$, and $p \in [1,\infty)$.
\begin{lemma}[A preliminary embedding result]\label{lma:Pre:Embedding}
    Let $u \in C^{1}(\Omega) \cap W^{1,p}(\Omega)$. Then
    \begin{equation*}
        [u]_{\frak{W}^{\beta,p}[\delta;q](\Omega)} \leq \frac{1}{1-\delta} [u]_{W^{1,p}(\Omega)}\,.
    \end{equation*}
\end{lemma}

\begin{proof}
    For $s = 1$ the proof is done exactly the same as in \cite{tian2017trace}. In this proof denote $\eta_\delta(\bx) = \delta q(d_{\p \Omega}(\bx))$. We have for any $\bz \in B(\bx,\eta_\delta(\bx))$
    \begin{equation*}
        |u(\bx+\bz) - u(\bx)|^p \leq |\bz|^p \int_0^1 |\grad u(\bx+t\bz)|^p \, \rmd t\,,
    \end{equation*}
    and so
    \begin{equation*}
        \begin{split}
            [u]_{\frak{W}^{\beta,p}[\delta;q](\Omega)}^p 
            &= \int_{\Omega} \int_{ B(0,\eta_\delta(\bx)) } \frac{C_{d,\beta,p}}{|\bz|^{\beta}} \frac{1}{\eta_\delta(\bx)^{d+p-\beta}} |u(\bx+\bz) - u(\bx)|^p \, \rmd \bz \, \rmd \bx \\
            &\leq \int_{\Omega} \int_{ B(0,\eta_\delta(\bx)) } \frac{C_{d,\beta,p}}{|\bz|^{\beta-p}} \frac{1}{\eta_\delta(\bx)^{d+p-\beta}} \int_0^1 |\grad u(\bx+t\bz)|^p \, \rmd t \, \rmd \bz \, \rmd \bx \\
            &= \int_{\Omega} \int_{ B(0,1) } \frac{C_{d,\beta,p}}{|\bz|^{\beta-p}} \int_0^1 |\grad u(\bx+t \eta_\delta(\bx) \bz)|^p \, \rmd t \, \rmd \bz \, \rmd \bx\,.
        \end{split}
    \end{equation*}
    Define the function $\bszeta_{t\bz}^{\delta}(\bx)$ as in \Cref{lma:CoordChange2}.
    Then since $\bszeta_{t\bz}^{\delta}$ is invertible on $\Omega$ with $\bszeta_{t\bz}^{\delta}(\Omega) \subset \Omega$, changing coordinates and using the identities in \Cref{lma:CoordChange2} give
    \begin{equation*}
        \begin{split}
            [u]_{\frak{W}^{\beta,p}[\delta;q](\Omega)}^p 
            &\leq \int_{\Omega} \int_{ B(0,1) } \frac{C_{d,\beta,p}}{|\bz|^{\beta-p}} \int_0^1 |\grad u(\bx+t \eta_\delta(\bx) \bz)|^p \, \rmd t \, \rmd \bz \, \rmd \bx \\
            &\leq \frac{1}{1- \delta} \int_{\Omega} \int_{ B(0,1) } \frac{C_{d,\beta,p}}{|\bz|^{\beta-p}} \int_0^1 |\grad u(\bx)|^p \, \rmd t \, \rmd \bz \, \rmd \bx \\
            &= \frac{1}{1- \delta} [u]_{W^{1,p}(\Omega)}^p\,.
        \end{split}
    \end{equation*}
\end{proof}

\begin{theorem}[Density of smooth functions, first version]\label{thm:Density:Prelim}
    $C^{k}(\Omega) \cap \frak{W}^{\beta,p}[\delta;q](\Omega)$ is dense in $\frak{W}^{\beta,p}[\delta;q](\Omega)$ for any $k \leq k_q$.
\end{theorem}

\begin{proof}[Proof of \Cref{thm:Density:Prelim}]
    Given $u \in \frak{W}^{\beta,p}[\delta;q](\Omega)$, we need to show that there exists a sequence $\{u_\veps \}_{\veps>0} \subset C^{k}(\Omega) \cap \frak{W}^{\beta,p}[\delta;q](\Omega)$ such that
    \begin{equation}
        \lim\limits_{\veps \to 0}  \Vnorm{u_\veps - u}_{\frak{W}^{\beta,p}[\delta;q](\Omega)} = 0\,.
    \end{equation}
    To this end, define $\lambda(\bx) = d_{\p \Omega}(\bx)$, with $\eta(\bx) = \eta[d_{\p \Omega},q](\bx)$, and for $\delta \in (0,\underline{\delta}_0)$ denote $\eta_\delta(\bx) = \delta \eta(\bx)$.
    Choose a function $\psi$ satisfying \eqref{Assump:Kernel} for $k = \infty$, and choose $\bar{\lambda}$ to satisfy \eqref{assump:Localization} for $k = \infty$.
    Define the constant $\wt{\veps}_0 := \min\{ \delta, \veps_0, \frac{1}{\bar{c}} \frac{\underline{\delta}_0 - \delta }{ \underline{\delta}_0 + \delta  } \}$, where $\bar{c}$ is the constant defined in \Cref{lma:CoordChange2}.
    Finally, define $\bar{\eta}(\bx) := \eta[\bar{\lambda},q](\bx)$, and for $\veps \in (0,\wt{\veps}_0)$, denote $\bar{\eta}_{\veps}(\bx) = \veps \bar{\eta}(\bx)$, and define the sequence $\{u_\veps \}_{\veps > 0}$ by
    \begin{equation*}
        u_{\veps}(\bx) := K_{\veps}[\bar{\lambda},q,\psi] u(\bx) = K_{\veps} u(\bx)\,.
    \end{equation*}
    Then $u_\veps \in C^{k}(\Omega)$ by \Cref{lma:SmoothnessOfConv}.
    Moreover, by \Cref{thm:ConvergenceOfConv} $u_\veps \to u$ in $L^p(\Omega)$ as $\veps \to 0$. 

    It remains to show that $\lim\limits_{\veps \to 0} [u_{\veps} - u]_{ \frak{W}^{\beta,p}[\delta;q](\Omega) } = 0 $.
    To this end, for $\veps \in (0,\delta)$ it follows from Jensen's inequality that
    \begin{equation*}
        \begin{split}
            |K_{\veps} u(\bx) - K_{\veps} u(\by)|^p &= \left| \int_{B(0,1)} \psi \left( |\bz| \right) ( u(\bx+ \bar{\eta}_{\veps}(\bx)\bz ) - u(\by +  \bar{\eta}_{\veps}(\by) \bz) \, \rmd \bz \right|^p \\
            &\leq  \int_{B(0,1)} \psi \left( |\bz| \right) \left|  u(\bx+  \bar{\eta}_\veps(\bx)\bz ) - u(\by +  \bar{\eta}_\veps(\by) \bz) \right|^p  \, \rmd \bz\,,
        \end{split}
    \end{equation*}
    and so
    \begin{equation*}
    \begin{split}
        &[K_{\veps} u]_{ \frak{W}^{\beta,p}[\delta;q](\Omega) }^p \\
        &\leq \int_{B(0,1)} \psi \left( |\bz| \right) \int_\Omega \int_{B(\bx, \eta_\delta(\bx)) } \frac{C_{d,\beta,p}}{|\bx-\by|^{\beta}} \frac{ \left|  u(\bx+ \bar{\eta}_\veps(\bx)\bz ) - u(\by +  \bar{\eta}_\veps(\by) \bz) \right|^p }{\eta_\delta(\bx)^{d+p-\beta}} \, \rmd \by \, \rmd \bx \, \rmd \bz\,.
    \end{split}
    \end{equation*}
    Now define the function $\bszeta_{\bz}^{\veps}(\bx) = \bx + \eta_\veps(\bx)\bz$ as in \Cref{lma:CoordChange2};
    since $\wt{\veps_0} < \veps_0$, we can apply this lemma to get
    \begin{equation*}
    \begin{split}
        &[K_{\veps} u]_{ \frak{W}^{\beta,p}[\delta;q](\Omega) }^p \\ 
        \leq & \frac{(1+\bar{c}\veps)^{d+p}}{(1-\bar{c} \veps)^{2}} \int_{B(0,1)} \psi \left( |\bz| \right) \int_\Omega \int_{ \Omega \cap \{ |\bszeta_{\bz}^{\veps}(\by) - \bszeta_{\bz}^{\veps}(\bx)| 
        \leq \frac{1+\bar{c} \veps}{1-\bar{c} \veps} \eta_\delta( \bszeta_{\bz}^{\veps}(\bx) ) \} } \frac{C_{d,\beta,p}}{|\bszeta_{\bz}^{\veps}(\by) - \bszeta_{\bz}^{\veps}(\bx)|^{\beta}} \\
        & \qquad \frac{ \left|  u( \bszeta_{\bz}^\veps(\bx) ) - u( \bszeta_{\bz}^\veps(\by)  ) \right|^p }{\eta_\delta( \bszeta_{\bz}^\veps(\bx)  )^{d+p-\beta}} \det \grad \bszeta_\bz^\veps(\bx) \det \grad \bszeta_\bz^\veps(\by) \, \rmd \by \, \rmd \bx \, \rmd \bz\,.
    \end{split}
    \end{equation*}
    Now, apply the change of variables $\bar{\by} = \bszeta_{\bz}^\veps(\by) $ and $\bar{\bx} = \bszeta_{\bz}^\veps(\bx)$. 
    The choice of $\wt{\veps}_0$ ensures that $\frac{1+\bar{c} \veps}{1-\bar{c} \veps} \delta < \underline{\delta}_0$, and so this permits us to apply \Cref{thm:InvariantHorizon}; in the notation of that theorem, we take $\delta_1 = \delta$ and $\delta_2 = \frac{1+\bar{c} \veps}{1-\bar{c} \veps} \delta$.
    Therefore we obtain that there exists a constant $C>0$ depending only on $d$, $\beta$, $p$, $\underline{\delta}_0$ and $\kappa_1$ such that
    \begin{equation*}
    \begin{split}
        &[K_{\veps} u]_{ \frak{W}^{\beta,p}[\delta;q](\Omega) }^p \\
        &\leq C \int_{B(0,1)} \psi \left( |\bz| \right) \int_\Omega \int_{B(\bar{\bx}, \eta_\delta(\bar{\bx})) } \frac{C_{d,\beta,p}}{|\bar{\by}-\bar{\bx}|^\beta } \frac{ \left|  u(\bar{\bx}) - u(\bar{\by}) \right|^p }{\eta_\delta(\bar{\bx})^{d+p-\beta}} \, \rmd \bar{\by} \, \rmd \bar{\bx} \, \rmd \bz\,.
    \end{split}
    \end{equation*}
    Therefore
    \begin{equation}\label{eq:ConvolutionDecreasesEnergy}
        [K_{\veps} u]_{ \frak{W}^{\beta,p}[\delta;q](\Omega) } \leq C [u]_{\frak{W}^{\beta,p}[\delta;q](\Omega)  } \,, \quad \text{ for } \veps < \wt{\veps}_0\,.
    \end{equation}
    
	Thanks to this estimate we can use continuity of the integral to conclude that for any $\tau > 0$ there exists $\varrho > 0$ independent of $\veps$ such that
	\begin{equation*}
		\int_{ \Omega } \int_{ \Omega  \cap \{ |\bx-\by| < \varrho \} } \gamma_{\beta,p}[\delta;q](\bx,\by) |K_{\veps} u(\bx) - K_{\veps}u(\by) - ( u(\bx) - u(\by) )|^p \, \rmd \by \, \rmd \bx < \tau\,.
	\end{equation*}
    By \eqref{eq:ComparabilityOfDistanceFxn1}, whenever $\varrho \leq |\bx-\by| \leq \eta_\delta(\bx)$ we have $\eta_\delta(\bx) \geq \varrho$ and $\eta_\delta(\by) \geq (1-\kappa_1 \delta) \varrho$, so therefore
	\begin{equation*}
		\begin{split}
		&\int_{ \Omega } \int_{ \Omega  \cap \{ |\bx-\by| \geq \varrho \} } \gamma_{\beta,p}[\delta;q](\bx,\by) |K_{\veps} u(\bx) - K_{\veps}u(\by) - ( u(\bx) - u(\by) )|^p \, \rmd \by \, \rmd \bx \\
		&\leq \frac{C}{ \varrho^{d+p} } \Vnorm{ K_{\veps} u - u }_{L^p(\Omega)}^p\,.
		\end{split}
	\end{equation*}
	Since $\Vnorm{ K_{\veps} u - u }_{L^p(\Omega)}^p \to 0$ as $\veps \to 0$, we conclude that
	\begin{equation*}
		\limsup\limits_{\veps \to 0} [K_{\veps} u - u]_{ \frak{W}^{\beta,p}[\delta;q](\Omega)}^p < \tau + C(\varrho) \limsup\limits_{\veps \to 0} \Vnorm{ K_{\veps} u - u }_{L^p(\Omega)}^p = \tau\,,
	\end{equation*}
	and so the convergence follows since follows since $\tau > 0$ is arbitrary.    
\end{proof}

A by-product of this proof is that, for $\veps \in \wt{\veps}_0$, $K_{\veps}[\bar{\lambda},q,\psi] : \frak{W}^{\beta,p}[\delta;q](\Omega) \to \frak{W}^{\beta,p}[\delta;q](\Omega)$ is a bounded operator.

\begin{proof}[Proof of \Cref{thm:Density}]
    Define $K_{\veps} u$ just as in \Cref{thm:Density:Prelim} for $\veps \ll \delta$; note that $K_{\veps} u \in W^{1,p}(\Omega)$ by \Cref{cor:RegularityOfHSOp} and \Cref{thm:InvariantHorizon}. Thus, by standard Sobolev extension we can assume that $K_{\veps} u \in W^{1,p}(\bbR^d)$. 
    Let $\varphi$ be a standard mollifier, and define 
    for $\bar{\veps} > 0$ small
    \begin{equation*}
        v_{\bar{\veps},\veps}(x) = \varphi_{\bar{\veps}} \ast K_{\veps} u(x)\,.
    \end{equation*}
    Then $v_{\bar{\veps},\veps} \in C^{\infty}(\overline{\Omega})$. Moreover, by \Cref{lma:Pre:Embedding}
    \begin{equation*}
        \lim\limits_{\bar{\veps} \to 0} \Vnorm{v_{\bar{\veps},\veps} - K_{\veps} u}_{\frak{W}^{\beta,p}[\delta;q](\Omega)}^p \leq \frac{1}{1-\delta} \lim\limits_{\bar{\veps} \to 0} \Vnorm{v_{\bar{\veps},\veps} - K_{\veps} u}_{W^{1,p}(\Omega)}^p = 0\,.
    \end{equation*}
    For each $n \in \bbN$, choose $\{ \veps_n \}_n$ to be a strictly decreasing sequence that satisfies $\Vnorm{ K_{\veps_n} u - u }_{\frak{W}^{\beta,p}[\delta;q](\Omega)} < \frac{1}{2n}$. Then for each each $n$ there exists $\bar{\veps}_n= \bar{\veps}_n(\veps_n)$ depending on $\veps_n$ such that $\Vnorm{ v_{\bar{\veps},\veps_n} - K_{\veps_n} u }_{\frak{W}^{\beta,p}[\delta;q](\Omega)} < \frac{1}{2n}$ for all $\bar{\veps} \leq \bar{\veps}_n$. We can choose the sequence $\{ \bar{\veps}_n \}_n$ to be strictly decreasing as well.
    Define $\{w_n\}_n = \{ v_{\bar{\veps}_n,\veps_n} \}_n$; we conclude with
    \begin{equation*}
    \begin{split}
        \Vnorm{w_n - u}_{\frak{W}^{\beta,p}[\delta;q](\Omega)} 
        &\leq \vnorm{ v_{\bar{\veps}_n(\veps_n),\veps_n} - K_{\veps_n} u }_{\frak{W}^{\beta,p}[\delta;q](\Omega)} + \Vnorm{ K_{\veps_n} u - u }_{\frak{W}^{\beta,p}[\delta;q](\Omega)} \\
        &\leq \sup_{\bar{\veps} \leq \bar{\veps}_n} \vnorm{ v_{\bar{\veps},\veps_n} - K_{\veps_n} u }_{\frak{W}^{\beta,p}[\delta;q](\Omega)} + \Vnorm{ K_{\veps_n} u - u }_{\frak{W}^{\beta,p}[\delta;q](\Omega)} \\
        &\leq \frac{1}{2n} + \Vnorm{ K_{\veps_n} u - u }_{\frak{W}^{\beta,p}[\delta;q](\Omega)} < \frac{1}{n}\,.
    \end{split}
    \end{equation*}
\end{proof}

\begin{corollary}[An embedding result]\label{cor:Embedding}
    Let $u \in W^{1,p}(\Omega)$. Then
    \begin{equation*}
        [u]_{\frak{W}^{\beta,p}[\delta;q](\Omega)} \leq \frac{1}{(1-\delta)^{1/p}} [u]_{W^{1,p}(\Omega)}\,.
    \end{equation*}
\end{corollary}

\subsection{Trace theorem}\label{subsec:Trace}

Given their properties established in \Cref{sec:HSEstimates}, we intuit that trace inequalities in the spirit of \cite{tian2017trace,du2022fractional,Foss2021} might be possibly established via the boundary-localized convolutions. This is indeed the case, and we demonstrate this in the following theorems.
The discussion in this subsection is under the same assumptions as in \Cref{subsec:Density}, 
except that only $p \in (1,\infty)$ is considered.
The trace theorems ensure that proper local boundary conditions can be imposed for the associated nonlocal problems.

\begin{proof}[Proof of \Cref{thm:TraceTheorem}]
        Let $\psi$ satisfy \eqref{Assump:Kernel} and let $\lambda$ satisfy \eqref{assump:Localization}, both with $k = \infty$, and define the boundary-localized convolution $K_\delta v= K_\delta[\lambda,q,\psi]v$.
	First, since $K_\delta  v \in W^{1,p}(\Omega)$, we have by \Cref{cor:RegularityOfHSOp}
	\begin{equation*}
		\Vnorm{T K_\delta u}_{W^{1-1/p,p}(\p \Omega)} \leq C \Vnorm{K_\delta u}_{W^{1,p}(\Omega)} \leq C \Vnorm{K_\delta u}_{\frak{W}^{\beta,p}[\delta;q](\Omega)}\,.
	\end{equation*}
	We now use \Cref{thm:Density}. Let $\{ u_n \} \subset C^{1}(\overline{\Omega})$ be a sequence converging to $u$ in $\frak{W}^{\beta,p}[\delta;q](\Omega)$. Then since $T K_\delta u_n = T u_n$ for all $n$ by \Cref{thm:ConvProp:UniformCase}
	\begin{equation*}
		\begin{split}
			\Vnorm{ T u_n - T u_m }_{W^{1-1/p,p}(\p \Omega)} 
			&= \Vnorm{ T K_\delta u_n - T K_\delta u_m }_{W^{1-1/p,p}(\p \Omega)} \\
			&\leq C \Vnorm{ u_n - u_m }_{\frak{W}^{\beta,p}[\delta;q]( \Omega)}\,.
		\end{split}
	\end{equation*}
	Therefore the bounded linear operator
 $T : \frak{W}^{\beta,p}[\delta;q](\Omega) \to W^{1-1/p,p}(\p \Omega)$ is well-defined.
\end{proof}

In the special case of $q(r) = r$ and $\beta = 0$, we recover the trace theorems proven in \cite{tian2017trace} for $p = 2$, and in \cite{du2022fractional} for general $p$ and -- in the notation of that work -- $s=1$.

Now that \Cref{thm:Density} gives a density result
for the nonlocal function space, the following theorem can be proved in the same way as \Cref{thm:TraceOfConv}.

\begin{theorem}\label{thm:Trace:NonlocalSpace}
    Suppose that $u \in \frak{W}^{\beta,p}[\delta;q](\Omega)$. Then 
	\begin{equation*}
		T K_{\delta} u = T u \quad \text{ in the sense of functions in } W^{1-1/p,p}(\p \Omega)\,.
	\end{equation*}
\end{theorem}

We also prove a Lebesgue point property. For this, we use the outer measure definition of $\scH^{s}$ (see for instance \cite{E15}), which is
    \begin{equation*}
        \scH^{s}(U) = \lim\limits_{\varrho \to 0^+} \scH^s_{\varrho}(U)\,,
    \end{equation*}
    where for $\varrho > 0$ 
    \begin{equation*}
        \scH^s_\varrho(U) := \inf \left\{ \sum_{n = 1}^\infty \diam(U_n)^s \, : \, \diam(U_n) \leq \varrho \text{ and } U \subset \bigcup_{n=1}^\infty U_n \right\}\,.
    \end{equation*}

\begin{theorem}\label{thm:LebPtProp}
    Let $T : \frak{W}^{\beta,p}[\delta;q](\Omega) \to W^{1-1/p,p}(\p \Omega)$ denote the trace operator. Define $(d-p)_+ = \max \{ d-p,0\}$, and for $s \geq 0$ denote $s$-dimensional Hausdorff measure by $\scH^{s}$.
    Then for $\scH^{(d-p)_+}$-almost every $\bx \in \p \Omega$ 
    \begin{equation*}
        \lim\limits_{\veps \to 0} \fint_{B(\bx,\veps)} |u(\by) - Tu(\bx)|^p \, \rmd \by = 0
    \end{equation*}
    for all $u \in \frak{W}^{\beta,p}[\delta;q](\Omega)$, i.e.
    \begin{equation*}
        Tu(\bx) = \lim\limits_{\veps \to 0} \fint_{B(\bx,\veps)} u(\by) \, \rmd \by\,, \qquad \text{ for } \scH^{(d-p)_+}\text{-a.e. } \bx \in \p \Omega\,.
    \end{equation*}
\end{theorem}

\begin{proof}
    First, we claim that for $\scH^{(d-p)_+}$-almost every $\bx_0 \in \p \Omega$,
    \begin{equation}\label{eq:GiustiLma}
        \limsup_{\veps \to 0} \frac{1}{\veps^{d-p}} \int_{\Omega \cap B(\bx_0,\veps)} \int_\Omega \gamma_{\beta,p}[\delta;q](\bx,\by) |u(\bx)-u(\by)|^p \, \rmd \by \, \rmd \bx = 0\,.
    \end{equation}
    To show this, we let $\tau > 0$, and define
    \begin{equation*}
        \cA_\tau := \left\{ \bx_0 \in \p \Omega \, : \, \limsup_{\veps \to 0} \int_{\Omega} \int_\Omega \frac{\mathds{1}_{B(\bx_0,\veps)} }{ \veps^{d-p} } \gamma_{\beta,p}[\delta;q](\bx,\by) |u(\bx)-u(\by)|^p \, \rmd \by \, \rmd \bx \geq \tau \right\}.
    \end{equation*}
    To establish \eqref{eq:GiustiLma} we show that $\scH^{(d-p)_+}(\cA_\tau) = 0$ for all $\tau > 0$.
    If $p \geq d$ this is satisfied trivially, so assume that $p < d$.
    Let $0 < \bar{\varrho} < \varrho$.
    Then for each $\bx_0 \in \cA_\tau$, there exists $0 < \veps_{\bx_0} < \bar{\varrho}$ such that 
    \begin{equation}\label{eq:GiustiLma:Pf1}
        \int_{\Omega} \int_\Omega \frac{\mathds{1}_{B(\bx_0,\veps_{\bx_0})} }{ \veps_{\bx_0}^{d-p} } \gamma_{\beta,p}[\delta;q](\bx,\by) |u(\bx)-u(\by)|^p \, \rmd \by \, \rmd \bx \geq \tau\,.
    \end{equation}
    Hence we can use the Vitali covering lemma to obtain a countable collection of disjoint balls $\{ B(\bx_n,\veps_n) \}_{n=1}^\infty$ such that $\veps_n \leq \bar{\varrho}$, \eqref{eq:GiustiLma:Pf1} is satisfied and $\cA_\tau \subset \cup_{n=1}^\infty B(\bx_n,5\veps_n)$.
    Therefore,
    \begin{equation*}
    \begin{split}
        \scH^{d-p}_{10\varrho}(A_\tau) &\leq \sum_{n=1}^\infty (C(d) 5\veps_n)^{d-p} \\
        &\leq \frac{C}{\tau} \sum_{n=1}^\infty \int_{\Omega} \int_\Omega \mathds{1}_{B(\bx_n,\veps_n)} \gamma_{\beta,p}[\delta;q](\bx,\by) |u(\bx)-u(\by)|^p \, \rmd \by \, \rmd \bx \\
        &\leq \frac{C}{\tau} \sum_{n=1}^\infty \int_{\Omega} \int_\Omega \mathds{1}_{ \{ \eta(\bx) \leq \bar{\varrho} \} } \gamma_{\beta,p}[\delta;q](\bx,\by) |u(\bx)-u(\by)|^p \, \rmd \by \, \rmd \bx\,.
    \end{split}
    \end{equation*}
    By taking $\bar{\varrho} \to 0$, we obtain that $\scH^{d-p}_{10\varrho}(A_\tau) = 0$ for all $\varrho > 0$. Taking $\varrho \to 0$ gives $\scH^{d-p}(A_\tau) = 0$, and so \eqref{eq:GiustiLma} is proved. 

    Now, for $\psi$ satisfying \eqref{Assump:Kernel} and $\lambda$ satisfying \eqref{assump:Localization}, define $K_{\delta} = K_{\delta}[\lambda,q]$. 
    In the same way as in the proof of \eqref{eq:KdeltaError} we obtain
    \begin{equation*}
        \begin{split}
        \fint_{\Omega \cap B(\bx_0,\veps)} &|u(\bx) - K_{\delta} u(\bx)|^p \, \rmd \bx \\
        &\leq \frac{C\delta^p }{\veps^{d-p}} \int_{\Omega \cap B(\bx_0,\veps)} \int_\Omega \gamma_{\beta,p}[\delta;q](\bx,\by) |u(\bx)-u(\by)|^p \, \rmd \by \, \rmd \bx\,,
        \end{split}
    \end{equation*}
    and so by \eqref{eq:GiustiLma}
    \begin{equation}\label{eq:LebPtProp:Pf1}
        \lim\limits_{\veps \to 0} \fint_{\Omega \cap B(\bx_0,\veps)} |u(\bx) - K_{\delta} u(\bx)|^p \, \rmd \bx = 0 \text{ for } \scH^{(d-p)_+}\text{-.a.e.} \, \bx_0 \in \p \Omega\,.
    \end{equation}

    Now, by \Cref{cor:RegularityOfHSOp} and by the Lebesgue point property for $W^{1,p}(\Omega)$ functions (see for instance \cite[Theorem 3.23]{giusti2003direct}) we have
    \begin{equation}\label{eq:LebPtProp:Pf2}
        \lim\limits_{\veps \to 0} \fint_{\Omega \cap B(\bx_0,\veps)} |K_{\delta} u(\bx) - T K_{\delta} u(\bx_0)|^p \, \rmd \bx = 0
    \end{equation}
    for $\scH^{(d-p)_+}$-almost every $\bx_0 \in \p \Omega$.
    
    Finally. by using \Cref{thm:Trace:NonlocalSpace}, \eqref{eq:LebPtProp:Pf1}, and \eqref{eq:LebPtProp:Pf2}, we get for $\scH^{(d-p)_+}$-a.e.\ $\bx_0 \in \p \Omega$,
    \begin{equation*}
        \begin{split}
           & \fint_{B(\bx_0,\veps)} |u(\bx) - Tu(\bx_0)|^p \, \rmd \bx 
            = \fint_{B(\bx_0,\veps)} |u(\bx) - T K_{\delta} u(\bx_0)|^p \, \rmd \bx \\
            &\quad \leq 2^{p-1} \fint_{B(\bx_0,\veps)} |K_{\delta} u(\bx) - T K_{\delta} u(\bx_0)|^p \, \rmd \bx
                + 2^{p-1} \fint_{B(\bx_0,\veps)} |u(\bx) - K_{\delta} u(\bx)|^p \, \rmd \bx \to 0 
        \end{split}
    \end{equation*}
as $\veps \to 0$, which concludes the proof.
\end{proof}

Note that it is much easier to prove that
\begin{equation*}
        Tu(\bx) = \lim\limits_{\veps \to 0} \fint_{B(\bx,\veps)} K_\delta u(\by) \, \rmd \by\,, \qquad \text{ for } \scH^{(d-p)_+}\text{-a.e. } \bx \in \p \Omega\,.
\end{equation*}

The trace theorems also give us an alternative way to define the homogeneous nonlocal spaces $\frak{W}^{\beta,p}_{0,\p \Omega_D}[\delta;q](\Omega)$ defined in \eqref{eq:HomNonlocSpDef}.

\begin{theorem}\label{thm:TraceZero}
    Let $1 < p < \infty$.
    Then a function $u$ belongs to $\frak{W}^{\beta,p}_{0, \p \Omega_D}[\delta;q](\Omega)$ if and only if $u \in \frak{W}^{\beta,p}[\delta;q](\Omega)$ and $T u = 0$ on $\p \Omega_D$.
\end{theorem}

\begin{proof}
The forward implication is clear from the continuity of the trace, so we need to prove the reverse implication.
Suppose that $u \in \frak{W}^{\beta,p}[\delta;q](\Omega)$ and $T u = 0$ on $\p \Omega_D$. 
Let $\psi$ satisfy \eqref{Assump:Kernel} and let $\lambda$ satisfy \eqref{assump:Localization} both with $k =\infty$, and for $0 < \veps < \delta$ 
define $K_\veps = K_\veps[\lambda,q,\psi]$. Then $T K_\veps u = 0$ on $\p \Omega_D$ by \Cref{thm:Trace:NonlocalSpace}, so $K_\veps u \in W^{1,p}_{0,\p \Omega_D}(\Omega)$ by \Cref{cor:RegularityOfHSOp} and \Cref{thm:InvariantHorizon}. Thus, for each $\veps$ there exists a sequence $\{ v_{\bar{\veps},\veps} \}_{\bar{\veps}} \subset C^1_c(\overline{\Omega} \setminus \p \Omega_D)$ that converges in $W^{1,p}(\Omega)$ to $K_\veps u$ as $\bar{\veps} \to 0$.
By \Cref{thm:Density:Prelim}, for each $n \in \bbN$, we can choose $\{ \veps_n \}_n$ to be a strictly decreasing sequence that satisfies $\Vnorm{ K_{\veps_n} u - u }_{\frak{W}^{\beta,p}[\delta;q](\Omega)} < \frac{1}{2n}$. Then for each each $n$ there exists $\bar{\veps}_n= \bar{\veps}_n(\veps_n)$ depending on $\veps_n$ such that $\Vnorm{ v_{\bar{\veps},\veps_n} - K_{\veps_n} u }_{\frak{W}^{\beta,p}[\delta;q](\Omega)} < \frac{1}{2n}$ for all $\bar{\veps} \leq \bar{\veps}_n$ thanks to \Cref{cor:Embedding}. We can choose the sequence $\{ \bar{\veps}_n \}_n$ to be strictly decreasing as well.
    Define $\{w_n\}_n = \{ v_{\bar{\veps}_n,\veps_n} \}_n$, then
    \begin{equation*}
    \begin{split}
        \Vnorm{w_n - u}_{\frak{W}^{\beta,p}[\delta;q](\Omega)} 
        &\leq \vnorm{ v_{\bar{\veps}_n(\veps_n),\veps_n} - K_{\veps_n} u }_{\frak{W}^{\beta,p}[\delta;q](\Omega)} + \Vnorm{ K_{\veps_n} u - u }_{\frak{W}^{\beta,p}[\delta;q](\Omega)} \\
        &\leq \sup_{\bar{\veps} \leq \bar{\veps}_n} \vnorm{ v_{\bar{\veps},\veps_n} - K_{\veps_n} u }_{\frak{W}^{\beta,p}[\delta;q](\Omega)} + \Vnorm{ K_{\veps_n} u - u }_{\frak{W}^{\beta,p}[\delta;q](\Omega)} \\
        &\leq \frac{1}{2n} + \Vnorm{ K_{\veps_n} u - u }_{\frak{W}^{\beta,p}[\delta;q](\Omega)} < \frac{1}{n}\,,
    \end{split}
    \end{equation*} 
    which concludes the proof.
\end{proof}

\subsection{Poincar\'e inequalities}\label{subsec:Poincare}
Our discussions here on nonlocal Poincar\'e inequalities
also follow the assumptions made in \Cref{subsec:Density}, though we note the special case of $p\in (1,\infty)$ in \cref{thm:PoincareRobin}.

\begin{theorem}\label{thm:PoincareDirichlet}
    There exists a constant $C_D=C_D(d,\beta,p,q,\Omega)$ such that for all $\delta < \underline{\delta}_0$,
	\begin{equation*}
		\Vnorm{u}_{L^p(\Omega)} \leq C_D [u]_{\frak{W}^{\beta,p}[\delta;q](\Omega)}\,,
        \quad \forall u \in \frak{W}^{\beta,p}_{0,\p \Omega_D}[\delta;q](\Omega)\,.
	\end{equation*}
\end{theorem}

\begin{proof}
    Let $\psi$ satisfy \eqref{Assump:Kernel} and let $\lambda$ satisfy \eqref{assump:Localization}, both with $k = \infty$, and define the boundary-localized convolution $K_\delta v= K_\delta[\lambda,q,\psi]v$. Let $\{u_n \} \subset C^{1}_c(\overline{\Omega} \setminus \p \Omega_D)$ be a sequence converging to $u$ in $\frak{W}^{\beta,p}[\delta;q](\Omega)$ as $n \to \infty$.
    By \Cref{cor:RegularityOfHSOp}, $K_{\delta} u_n \in W^{1,p}(\Omega)$, and it follows from a slight modification of the proof of \Cref{lma:SupportOfConv} that $K_\delta u_n$ has support compactly contained in $\overline{\Omega} \setminus \p \Omega_D$. 
    Therefore $K_{\delta} u_n \in W^{1,p}_{0,\p \Omega_D}(\Omega)$, and we can apply the classical Poincar\'e inequality:
	\begin{equation}\label{eq:PoincareDirichlet:Pf1}
		\Vnorm{ K_{\delta} u_n}_{L^p(\Omega)} \leq C(p,\Omega) \Vnorm{\grad K_{\delta} u_n}_{L^{p}(\Omega)}\,.
	\end{equation}
	By \Cref{cor:RegularityOfHSOp},
	$\Vnorm{\grad K_{\delta} u_n}_{L^{p}(\Omega)} \leq C [u_n]_{\frak{W}^{\beta,p}[\delta;q](\Omega)}$.
	Then by \eqref{eq:KdeltaError},
	\begin{equation*}
		\begin{split}
			\Vnorm{u_n}_{L^p(\Omega)} \leq \Vnorm{ K_{\delta} u_n}_{L^p(\Omega)} + \Vnorm{ u_n - K_{\delta} u_n}_{L^p(\Omega)} \leq (C + C \delta) [u_n]_{\frak{W}^{\beta,p}[\delta;q](\Omega)}\,.
		\end{split}
	\end{equation*}
    The result follows by taking $n \to \infty$.
\end{proof}

\begin{theorem}\label{thm:PoincareNeumann}
    There exists a constant $C_N=C_N(d,\beta,p,q,\Omega)$ such that for all $\delta < \underline{\delta}_0$,
	\begin{equation*}
		\Vnorm{u - (u)_{\Omega} }_{L^p(\Omega)} \leq C_N [u]_{\frak{W}^{\beta,p}[\delta;q](\Omega)}\,, \qquad \forall u \in \frak{W}^{\beta,p}[\delta;q](\Omega)\,.
	\end{equation*}
\end{theorem}
	
\begin{proof}
    Let $\psi$ satisfy \eqref{Assump:Kernel} and let $\lambda$ satisfy \eqref{assump:Localization} with $k = \infty$ for both,  and define the boundary-localized convolution $K_\delta v= K_\delta[\lambda,q,\psi]v$. 
    By \Cref{cor:RegularityOfHSOp}, $K_{\delta} u \in W^{1,p}(\Omega)$.
    Therefore we can apply the classical Poincar\'e inequality:
    \begin{equation}\label{eq:PoincareNeumann:Pf1}
		\Vnorm{ K_{\delta} u - (K_{\delta} u)_\Omega }_{L^p(\Omega)} \leq C(p,\Omega) \Vnorm{\grad K_{\delta} u}_{L^{p}(\Omega)}\,.
    \end{equation}
    By \Cref{cor:RegularityOfHSOp},
	$\Vnorm{\grad K_{\delta} u}_{L^{p}(\Omega)} \leq C [u]_{\frak{W}^{\beta,p}[\delta;q](\Omega)}$.
    Now, recall the definition of $\Psi_{\delta}$ in \eqref{eq:KernelIntegral}, and note that
    \begin{equation*}
        (K_{\delta} u)_{\Omega} = \fint_{\Omega} K_{\delta} u(\bx) \, \rmd \bx = \fint_{\Omega} \int_{\Omega} \psi_{\delta}(\bx,\by) u(\by) \, \rmd \by \, \rmd \bx = \fint_{\Omega} \Psi_{\delta}(\by) u(\by) \, \rmd \by = (\Psi_{\delta} u)_\Omega\,,
    \end{equation*}
    and so by \eqref{eq:KdeltaError} and \eqref{eq:PoincareNeumann:Pf1}
    \begin{equation*}
        \begin{split}
        \Vnorm{u - (\Psi_\delta u)_{\Omega}}_{L^p(\Omega)} \leq \Vnorm{ K_{\delta} u - (K_{\delta} u)_{\Omega} }_{L^p(\Omega)} + \Vnorm{ u - K_{\delta} u}_{L^p(\Omega)} \leq C [u]_{\frak{W}^{\beta,p}[\delta;q](\Omega)}\,.
        \end{split}
    \end{equation*}
    Finally, by Jensen's inequality and by \eqref{eq:KdeltaError}
    \begin{equation*}
        \begin{split}
            \Vnorm{u - (u)_{\Omega}}_{L^p(\Omega)} 
            &\leq \Vnorm{u - (\Psi_\delta u)_{\Omega}}_{L^p(\Omega)} + \Vnorm{(\Psi_\delta u)_{\Omega} - (u)_{\Omega} }_{L^p(\Omega)} \\
            &\leq C [u]_{\frak{W}^{\beta,p}[\delta;q](\Omega)} + \left( \fint_{\Omega} \int_{\Omega} |K_{\delta} u(\by) - u(\by)|^p \, \rmd \by  \, \rmd \bx \right)^{1/p} \\
            &\leq C [u]_{\frak{W}^{\beta,p}[\delta;q](\Omega)} + \Vnorm{ u - K_{\delta} u}_{L^p(\Omega)} \leq C [u]_{\frak{W}^{\beta,p}[\delta;q](\Omega)}\,.
        \end{split}
    \end{equation*}
\end{proof}

\begin{remark}
    Note that all of the above Poincar\'e constants are constructed, and not given by a contradiction argument.
\end{remark}

\begin{theorem}\label{thm:PoincareRobin}
    Let $1 < p < \infty$.
    There exists a constant $C_R > 0$ such that for all $\delta < \underline{\delta}_0$
	\begin{equation*}
		\Vnorm{v}_{L^{p}(\Omega)}^p \leq C_R \left( [u]_{\frak{W}^{\beta,p}[\delta;q](\Omega)}^p + \int_{\p \Omega_R} |Tu|^p \, \rmd \sigma \right)\,, \qquad \forall u \in \frak{W}^{\beta,p}[\delta;q](\Omega)\,.
	\end{equation*}
\end{theorem}
	
\begin{proof}
    First, it is straightforward to prove via the compact embedding of $W^{1,p}(\Omega)$ into $L^p(\Omega)$ that there exists a constant $\Lambda > 0 $ such that 
    \begin{equation}\label{eq:PoincareRobin:Pf1}
        \Vnorm{v}_{L^{p}(\Omega)}^p \leq \Lambda \left( \Vnorm{\grad v}_{L^p(\Omega)}^p + \int_{\p \Omega_R} |Tv|^p \, \rmd \sigma \right) \qquad \forall v \in W^{1,p}(\Omega)\,.
    \end{equation}
    Now let $\psi$ satisfy \eqref{Assump:Kernel} and let $\lambda$ satisfy \eqref{assump:Localization} with $k = \infty$, and define the boundary-localized convolution $K_\delta v= K_\delta[\lambda,q,\psi]v$. 
    By \Cref{cor:RegularityOfHSOp}, $K_{\delta} u \in W^{1,p}(\Omega)$, and 
    therefore we can apply \eqref{eq:PoincareRobin:Pf1}:
    \begin{equation*}
		\Vnorm{K_\delta u}_{L^p(\Omega)}^p \leq \Lambda \left( \Vnorm{\grad K_\delta u}_{L^p(\Omega)}^p + \int_{\p \Omega_R} |T K_\delta u|^p \, \rmd \sigma \right)\,.
    \end{equation*}
    By \Cref{cor:RegularityOfHSOp}, $\Vnorm{\grad K_{\delta} u}_{L^{p}(\Omega)} \leq C [u]_{\frak{W}^{\beta,p}[\delta;q](\Omega)}$, and by \Cref{thm:Trace:NonlocalSpace} $T K_\delta u = T u $. Therefore by \eqref{eq:PoincareRobin:Pf1} and \eqref{eq:KdeltaError} 
    \begin{equation*}
        \begin{split}
        \Vnorm{u}_{L^p(\Omega)}
        &\leq \Vnorm{ K_{\delta} u}_{L^p(\Omega)} + \Vnorm{ u - K_{\delta} u}_{L^p(\Omega)} \\
        &\leq \Lambda \left( \Vnorm{\grad K_\delta u}_{L^p(\Omega)}^p + \int_{\p \Omega_R} |T K_\delta u|^p \, \rmd \sigma \right) + C [u]_{\frak{W}^{\beta,p}[\delta;q](\Omega)} \\
        &\leq C_R \left( [u]_{\frak{W}^{\beta,p}[\delta;q](\Omega)}^p + \int_{\p \Omega_R} |Tu|^p \, \rmd \sigma \right)\,,
        \end{split}
    \end{equation*}
    which concludes the proof.
\end{proof}

\begin{remark}
    More general Poincar\'e inequalities can be obtained using the same methods. Indeed, let $V$ be a weakly closed subset of $\frak{W}^{\beta,p}[\delta;q](\Omega)$ such that $V \cap \bbR = \{ 0 \}$. Then a Poincar\'e inequality holds on $V$. Thanks to the heterogeneous localization properties, $V$ can possibly be characterized either by lower-order terms, or by terms depending only on boundary values. In addition,
    Poincar\'e inequalities for more general forms on the right-hand side, for instance $[u]_{\frak{W}^{\beta,p}[\delta;q](\Omega}^p + \int_{\p \Omega_D} |Tu|^m \, \rmd \sigma$ for some exponent $m \in [1,p]$, can be obtained.
\end{remark}

\subsection{Compactness for a fixed bulk
horizon parameter}

We continue our discussion with considering compact embeddings for nonlocal spaces, under the assumptions \eqref{assump:beta}, \eqref{assump:NonlinearLocalization}, and $1 \leq p < \infty$.

If $\beta < d$, then it is straightforward to see that the embedding $\frak{W}^{\beta,p}[\delta;q](\Omega) \hookrightarrow L^p(\Omega)$ is not compact. 
Indeed, for any cube $Q \Subset \Omega$ with sides parallel to the axes, let $\{ u_n \}_n$ be the standard Fourier basis for $L^{2}(Q)$. Extending the $u_n$ to all of $\Omega$ by $0$, we then have
$[u_n]_{ \frak{W}^{\beta,p}[\delta;q](\Omega) } \leq C(Q,\Omega,\Vnorm{u_n}_{L^\infty(Q)}) \leq C$ independent of $n$, since $|\bx|^{-\beta} \in L^1_{loc}(\bbR^d)$. However, $u_n$ does not converge strongly in $L^p(\Omega)$.

On the other hand, for $\beta > d$, the nonlocal function space actually contains fractional Sobolev-Slobodeckij spaces, and thus inherits their embedding properties. The proof relies on embedding properties of Sobolev spaces and the estimate \Cref{thm:diffuKu:Fractional}.

\begin{theorem}\label{thm:Embedding:Fractional}
    Assume that $\beta > d$.
    Then there exists a constant $C >0$ depending only on $d$, $\beta$, $p$, $q$, and $\Omega$ such that for all $\delta$ satisfying \eqref{eq:HorizonThreshold2} 
    \begin{equation*}
        \Vnorm{u}_{W^{(\beta-d)/p,p}(\Omega)} \leq 
        C \Vnorm{u}_{ \frak{W}^{\beta,p}[\delta;q](\Omega) }, \quad \forall u \in \frak{W}^{\beta,p}[\delta;q](\Omega).
    \end{equation*}
\end{theorem}

\begin{proof}
    Let $\psi$ satisfy \eqref{Assump:Kernel} for $k_\psi \geq 1$, with $K_\delta = K_\delta[d_{\p \Omega},q,\psi]$.
    First, by \Cref{thm:diffuKu:Fractional}
    the estimate
    \begin{equation*}
        [u-K_\delta u]_{W^{(\beta-d)/p,p}(\Omega)} \leq C \delta^{\frac{d+p-\beta}{p}} [u]_{\frak{W}^{\beta,p}[\delta;q](\Omega)}
    \end{equation*}
    holds.
    Now we use the embedding of Sobolev spaces with varying differentiability index along with \Cref{cor:RegularityOfHSOp} to get
    \begin{equation*}
        \Vnorm{K_\delta u}_{W^{(\beta-d)/p,p}(\Omega)} \leq C(d,\beta,p) \Vnorm{K_\delta u}_{W^{1,p}(\Omega)} \leq C [K_\delta u]_{\frak{W}^{\beta,p}[\delta;q](\Omega)}\,.
    \end{equation*}
    Combining these two estimates gives the result.
\end{proof}

As a result, all of the embeddings that hold for fractional Sobolev spaces hold for the nonlocal space.
So that we can use it for the variational problems, we state explicitly the $L^p$-space embedding.
\begin{theorem}\label{thm:Compactness:FixedSpace}
    For $\beta \in (d,d+p)$, let $p^*_\beta$ denote
	 \begin{equation}\label{eq:EmbeddingExponent}
	 	p^*_{\beta} := 
	 	\begin{cases}
	 	\frac{dp}{2d-\beta}, &\text{ if } d < \beta < 2d\,, \\
	 	\text{any exponent } < \infty, &\text{ if } \beta \geq 2d\,.
	 	\end{cases}
	 \end{equation}
Then for $\delta < \bar{\delta}_0$ the embedding $\frak{W}^{\beta,p}[\delta;q](\Omega) \hookrightarrow L^{p^*_\beta}(\Omega)$ is continuous, and for any $q < p^*_\beta$ the embedding $\frak{W}^{\beta,p}[\delta;q](\Omega) \hookrightarrow L^q(\Omega)$ is compact. 
\end{theorem}

\subsection{Asymptotic compactness in the local limit}\label{sec:Compactness}
In the following section we prove a general compactness result in the local limit, which encompasses \Cref{thm:Intro:Compactness}. We take all the assumptions of \Cref{subsec:Density}.
Note that \cref{thm:Compactness} and \cref{lma:GammaConv:WeakConvOfMoll} require the additional assumption 
\eqref{eq:assump:rho:q}.

We remark that in the case $p = 1$, all of the results of this section hold for the space $BV(\Omega)$, functions of bounded variation, instead of the space $W^{1,1}(\Omega)$. The proofs are almost exactly the same; the differences are the same as in the proofs contained in \cite{BBM,ponce2004new}. Since we do not consider problems associated with functionals defined in $BV$ spaces in this work, the precise statements are omitted.

\begin{theorem}\label{thm:LocalizationOfSeminorm}
    Let $ p \in (1, \infty)$. Then 
    \begin{equation*}
    \lim\limits_{\delta \to 0} [u]_{\frak{V}^{\beta,p}[\delta;q;\rho,\lambda](\Omega)}^p =
         	\begin{cases}
          \Vnorm{\grad u}_{L^p(\Omega)}^p\,, & 
          \text{if }  u \in W^{1,p}(\Omega),\\
           + \infty, & \text{if } u \in L^p(\Omega) \setminus W^{1,p}(\Omega).
           \end{cases}
    \end{equation*}    
    Moreover, if a sequence $\{u_\delta\}_\delta$ converges to $u$ in $C^2(\overline{V})$ for any $V \Subset \Omega$ as $\delta \to 0$, then
    \begin{equation*}
        \lim\limits_{\delta \to 0} \int_{V} \int_{V} \gamma_{\beta,p}[\delta;q;\rho,\lambda](\bx,\by) |u(\by)-u(\bx)|^p \, \rmd \by \, \rmd \bx = \int_{V} |\grad u(\bx)|^p \, \rmd \bx\,.
    \end{equation*}
\end{theorem}

\begin{proof}
    The proof of the first statement follows exactly the same steps as \cite[Theorem 1.1]{ponce2004new}, and the proof of the second statement follows exactly the same steps as \cite[Proposition 4.1, Remarks 4.1 and 4.2]{ponce2004new}. The heterogeneous localization $\eta(\bx)$ gives no additional difficulty.
\end{proof}

\begin{theorem}\label{thm:Compactness}
  Assume additionally that \eqref{eq:assump:rho:q} holds. Let    
    $\delta = \{\delta_n \}_{n \in \bbN}$ be a sequence that converges to $0$ and $\{ u_\delta \}_\delta \subset \frak{W}^{\beta,p}[\delta;q](\Omega)$ be a sequence such that for constants $B$ and $C$ independent of $\delta$,
    $$
    \sup_{\delta > 0} \Vnorm{u_\delta}_{L^p(\Omega)} \leq C < \infty\; \text{ and }\;
    \sup_{\delta > 0} [u_\delta]_{\frak{V}^{\beta,p}[\delta;q;\rho,\lambda](\Omega)} := B < \infty\,.
    $$
    Then $\{ u_\delta \}_\delta$ is precompact in the strong topology of $L^p(\Omega)$. Moreover, if $p > 1$ any limit point $u$ belongs to $W^{1,p}(\Omega)$ with $\Vnorm{\grad u}_{L^{p}(\Omega)} \leq B$. 
\end{theorem}

\begin{proof}
    It suffices to show that a subsequence of $\{u_\delta\}$ is Cauchy in $\frak{W}^{\beta,p}[\delta;q](\Omega)$. Choose $\psi$ to satisfy \eqref{Assump:Kernel} for $k=k_\psi \geq 1$, and then define
    $K_{\delta} u = K_\delta[\lambda,q,\psi] u$ accordingly.
    First, we use \Cref{thm:diffuKu} and \Cref{thm:EnergySpaceIndepOfKernel} to see that
    \begin{equation*}
        \Vnorm{ u_\delta -  K_{\delta} u_\delta }_{L^{p}(\Omega)} \leq C\delta [u_{\delta}]_{\frak{W}^{\beta,p}[\delta;q](\Omega)} \leq C B\delta\,.
    \end{equation*}
    Next, by \Cref{cor:RegularityOfHSOp}
    \begin{equation*}
        [ K_{\delta} u_\delta ]_{W^{1,p}(\Omega)} \leq C [ u_\delta ]_{\frak{W}^{\beta,p}[\delta;q](\Omega)}\,.
    \end{equation*}
    Therefore the sequence $\{ K_{\delta_n} u_{\delta_n} \}_{n \in \bbN}$ is bounded in $W^{1,p}(\Omega)$, hence is precompact in the strong topology of $L^p(\Omega)$.
    So for a convergent subsequence $\{ K_{\delta_n} u_{\delta_n} \}$ (not relabeled), we have for $n$, $m \in \bbN$
    \begin{equation*}
    \begin{split}
        \Vnorm{ u_{\delta_n} - u_{\delta_m} }_{L^p(\Omega)} 
        &\leq \Vnorm{ K_{\delta_n} u_{\delta_n} - u_{\delta_n} }_{L^p(\Omega)} + \Vnorm{ K_{\delta_m} u_{\delta_m} - u_{\delta_m} }_{L^p(\Omega)} \\
            &\qquad + \Vnorm{ K_{\delta_n} u_{\delta_n} - K_{\delta_m} u_{\delta_m} }_{L^p(\Omega)} \\
        &\leq CB( \delta_m + \delta_n) + \Vnorm{ K_{\delta_n} u_{\delta_n} - K_{\delta_m} u_{\delta_m} }_{L^p(\Omega)} 
    \end{split}
    \end{equation*}
which approaches $0$ as $\min \{m,n\} \to \infty$. Thus $ \{u_{\delta_n} \}$ is also convergent.
        
    To see that any limit point $u$ belongs to $W^{1,p}(\Omega)$, we use an argument similar to the one used to prove \Cref{thm:Density}. Let $\psi$ satisfy \eqref{Assump:Kernel} and $\bar{\lambda}$ satisfy \eqref{assump:Localization}, both for $k = \infty$, and define $K_\veps u= K_\veps[\bar{\lambda},q,\psi] u$ accordingly, where $\veps \in (0,\veps_0)$, $\veps_0$ as defined in \Cref{lma:CoordChange2}, with $\bar{\eta}_\veps(\bx) = \veps \bar{\eta}(\bx) = \veps q(\bar{\lambda}(\bx))$. 
    Define the function $\bszeta_\bz^\veps(\bx) = \bx + \bar{\eta}_\veps(\bx) \bz$ as in \Cref{lma:CoordChange2}.
    Now define $\rho_\veps(|\bz|) := \rho \big( \frac{1+\bar{c} \veps}{1-\bar{c}\veps} |\bz| \big)$; by definition of the nonlocal seminorm we have for any $v \in \frak{W}^{\beta,p}[\delta;q](\Omega)$
    \begin{equation*}
    \begin{split}
        &[v]_{ \frak{V}^{\beta,p}[\delta;q;\rho_\veps,\lambda](\Omega) }^p \\
        &= \frac{1}{A_\rho} \left( \frac{1+\bar{c} \veps}{ 1- \bar{c}\veps } \right)^{d+p-\beta} \int_\Omega \int_\Omega \rho \left( \frac{1+\bar{c} \veps}{1-\bar{c}\veps} \frac{|\bx-\by|}{\eta_\delta(\bx)} \right) \frac{|v(\bx)-v(\by)|^p}{ |\bx-\by|^{\beta} \eta_\delta(\bx)^{d+p-\beta} } \, \rmd \by \, \rmd \bx\,,
    \end{split}
    \end{equation*}
    where $A_\rho := \frac{\bar{\rho}_{\beta,p}}{\overline{C}_{d,p}}$. Then by Jensen's inequality
    \begin{equation*}
    \begin{split}
        &[K_{\veps} u_\delta]_{ \frak{V}^{\beta,p}[\delta;q;\rho_\veps,\lambda](\Omega^{\veps;\lambda,q}) }^p 
        \leq
        \frac{1}{A_\rho} \left( \frac{1+\bar{c} \veps}{ 1- \bar{c}\veps } \right)^{d+p-\beta} 
        \int_{B(0,1)}  \int_{ \Omega^{\veps;\lambda,q} } \int_{\Omega^{\veps;\lambda,q} } \psi \left( |\bz| \right)  \\
        &\qquad\qquad  \rho_\veps \left( \frac{|\by-\bx|}{\eta_\delta(\bx)} \right) \frac{ \left|  u_\delta( \bszeta_{\bz}^{\veps}(\bx) ) - u_\delta( \bszeta_{\bz}^{\veps}(\by) ) \right|^p }{|\bx-\by|^{\beta} \eta_\delta(\bx)^{d+p-\beta}} \, \rmd \by \, \rmd \bx \, \rmd \bz\,.
    \end{split}
    \end{equation*}
    By the identities in \Cref{lma:CoordChange2}, we obtain that for $\delta < \veps < \veps_0$, $\bszeta_{\bz}^{\veps}(\Omega^{\veps;\lambda,q}) \subset \Omega^{(1-\bar{c}\veps)\veps;\lambda,q}$, and that 
    \begin{equation*}
    \begin{split}
        &[K_{\veps} u_\delta]_{ \frak{V}^{\beta,p}[\delta;q;\rho_\veps,\lambda](\Omega^{\veps;\lambda,q}) }^p  \\ 
        &\leq \frac{1}{A_\rho}
        \frac{ (1+\bar{c}\veps)^{2(d+p)-\beta} }{ (1-\bar{c} \veps)^{2} }
        \int_{B(0,1)} \psi \left( |\bz| \right) \int_{ \Omega^{(1-\bar{c}\veps)\veps;\lambda,q } } \int_{ \Omega^{(1-\bar{c}\veps)\veps;\lambda,q } } 
        \rho \left( \frac{ |\bszeta_{\bz}^{\veps}(\by) - \bszeta_{\bz}^{\veps}(\bx)| }{\eta_\delta(\bszeta_{\bz}^{\veps}(\bx))}  \right) 
        \\
        &\qquad \frac{ \left|  u_\delta( \bszeta_{\bz}^\veps(\bx) ) - u_\delta( \bszeta_{\bz}^\veps(\by)  ) \right|^p }{|\bszeta_{\bz}^{\veps}(\by) - \bszeta_{\bz}^{\veps}(\bx)|^{\beta} \eta_\delta(\bszeta_{\bz}^{\veps}(\bx))^{d+p-\beta}}  \det \grad \bszeta_{\bz}^{\veps}(\bx) \det \grad \bszeta_{\bz}^{\veps}(\by)  \, \rmd \by \, \rmd \bx \, \rmd \bz\,,
    \end{split}
    \end{equation*}
    where we additionally used that $\rho$ is nonincreasing.
    Therefore we apply the change of variables $\bar{\by} = \bszeta_{\bz}^\veps(\by) $, $\bar{\bx} = \bszeta_{\bz}^\veps(\bx)$, and obtain for any for any $ \delta < \veps$
    \begin{equation*}
    \begin{split}
        &[K_{\veps} u_\delta]_{ \frak{V}^{\beta,p}[\delta;q;\rho_\veps,\lambda](\Omega^{\veps;\lambda,q}) }^p  \\ 
        &\leq 
        \frac{1}{A_\rho}
        \frac{ (1+\bar{c}\veps)^{2(d+p)-\beta} }{ (1-\bar{c} \veps)^{2} } 
        \int_{ \Omega } \int_{ \Omega }
        \rho \left( \frac{ |\bar{\by} - \bar{\bx}| }{ \eta_\delta(\bar{\bx})}  \right) \frac{ \left|  u_\delta(\bar{\bx}) - u_\delta(\bar{\by}) \right|^p }{|\bar{\by} - \bar{\bx}|^{\beta} \eta_\delta(\bar{\bx})^{d+p-\beta}} \rmd \bar{\by} \, \rmd \bar{\bx}\,.
    \end{split}
    \end{equation*}
    In summary,
    \begin{equation*}
        [K_{\veps} u_\delta]_{ \frak{V}^{\beta,p}[\delta;q;\rho_\veps,\lambda](\Omega^{\veps;\lambda,q}) }^p \leq \frac{ (1+\bar{c}\veps)^{2(d+p)-\beta} }{ (1-\bar{c} \veps)^{2} } [u_\delta]_{ \frak{V}^{\beta,p}[\delta;q;\rho,\lambda](\Omega) }^p\,.
    \end{equation*}
    Now for any fixed $\veps > 0$, the sequence $\{ K_\veps u_{\delta_n} \}_n$ converges to $K_{\veps} u$ in $C^2(\overline{\Omega^{ \veps;\lambda,q}})$ as $\delta_n \to 0$, since $k_q \geq 2$ and $\Omega^{\veps;\lambda,q} \Subset \Omega$.
    Therefore we can use \Cref{thm:LocalizationOfSeminorm} when taking $\delta_n \to 0$ in the previous inequality to get
    \begin{equation*}
        \int_{\Omega^{ \veps;\lambda,q}} |\grad K_\veps u|^p \, \rmd \bx \leq 
        \frac{ (1+\bar{c}\veps)^{2(d+p)-\beta} }{ (1-\bar{c} \veps)^{2} } B^p\,.
    \end{equation*}
	This inequality holds uniformly in $\veps$, so the result follows by taking $\veps \to 0$.
\end{proof}

\begin{remark}
    In the case $\beta > d$, we can use \Cref{thm:diffuKu:Fractional} in place of \eqref{eq:KdeltaError} in the proof, and obtain that $\{u_\delta\}_\delta$ is precompact in the fractional Sobolev space $W^{(\beta-d)/p,p}(\Omega)$. This compactness allows us to obtain variational convergence of more general energy functionals, but to illustrate the ideas in this work we are content to consider only semilinear functionals.
\end{remark}

One consequence of the compactness result is the $W^{1,p}(\Omega)$-weak convergence of boundary-localized convolutions, which will be instrumental in the analysis of the nonlocal-to-local limit of the variational problems.

\begin{lemma}\label{lma:GammaConv:WeakConvOfMoll}
    Let $p \in (1,\infty)$,
    $\psi$ satisfy \eqref{Assump:Kernel} with $k_\psi \geq 1$,
    and assume \eqref{eq:assump:rho:q}.
    Let $\{\delta\} = \{\delta_n \}_{n \in \bbN}$ be a sequence converging to $0$. 
    Suppose that
    $\sup_{\delta > 0 } \Vnorm{ u_{\delta} }_{ \frak{W}^{\beta,p}[\delta;q](\Omega) } < \infty$.  Then there exists a subsequence $\{u_{\delta'} \}_{\delta'}$ and a function $u \in W^{1,p}(\Omega)$ such that $K_{\delta'} u_{\delta'} \rightharpoonup u$ weakly in $W^{1,p}(\Omega)$.
    If additionally there exists $u \in W^{1,p}(\Omega)$ such that the entire sequence $u_\delta \to u$ strongly in $L^p(\Omega)$, then the whole sequence $K_\delta u_\delta \rightharpoonup u$ weakly in $W^{1,p}(\Omega)$.
\end{lemma}

\begin{proof}
    We select the subsequence $\{ u_{\delta'} \} \subset \{ u_\delta \}$ to be one with a strong-$L^p(\Omega)$ limit $u$ as in \Cref{thm:Compactness} (if $u_\delta \to u$ in $L^p(\Omega)$ then choose $\{ u_{\delta'} \} \subset \{ u_\delta \}$). By \Cref{thm:ConvergenceOfConv:W1p} it suffices to show that $K_{\delta'} (u_{\delta'} - u) \rightharpoonup 0$ weakly in $W^{1,p}(\Omega)$ as $\delta \to 0$.
    By \Cref{cor:RegularityOfHSOp}, \eqref{eq:ConvEst:Lp}, \eqref{eq:ConvEst:W1p}, and \Cref{thm:Compactness} we have
    \begin{equation*}
    \begin{split}
        \Vnorm{ K_{\delta'} (u_{\delta'} - u) }_{W^{1,p}(\Omega)} 
        &\leq \Vnorm{ K_{\delta'} u_{\delta'} }_{W^{1,p}(\Omega)} + \Vnorm{ K_{\delta'} u }_{W^{1,p}(\Omega)} \\
        &\leq C \Vnorm{ u_{\delta'} }_{ \frak{W}^{\beta,p}[\delta';q](\Omega) } + \Vnorm{ u }_{W^{1,p}(\Omega)} \leq C\,,
    \end{split}
    \end{equation*}
    hence there exists at least one convergent subsubsequence.
    Now let $K_{\delta''} (u_{\delta''} - u)$ be any subsubsequence converging weakly in $W^{1,p}(\Omega)$ to a function $v$. 
    However, $\Vint{ K_{\delta''} (u_{\delta''} - u) , \varphi } \to 0$ for any $\varphi \in L^{p'}(\Omega)$ since $u_{\delta'} \to u$ strongly in $L^p(\Omega)$, and so it follows that $v = 0$ since weak limits are unique.
\end{proof}

\section{Existence of minimizers to the minimization problems}\label{sec:varprob}

In this section we analyze the variational problems.
We now describe the assumptions on $\cG_{\beta>d}$ using the notation from the previous section. For a fixed $m \in (1,p^*_\beta)$, where $p^*_\beta$ is as in \eqref{eq:EmbeddingExponent}, we assume that $\cG_{\beta>d} : L^m(\Omega) \to \bbR$ is $L^m(\Omega)$-strongly continuous (but possibly nonconvex) and satisfies, for $c > 0$, $C >0$, $\theta \in (0,1)$ and $\Theta >0$
\begin{equation}\label{eq:LowerOrderTerm2}
\begin{gathered}
    - c \chi_{d,\beta}  (1+\Vnorm{u}_{L^{m}(\Omega)}^{\theta p}) \leq 
   \cG_{\beta>d}
    (u) \leq C \chi_{d,\beta}  ( 1 + \Vnorm{u}_{L^{m}(\Omega)}^{\Theta p})\,,  
\end{gathered}
\end{equation}
where the constant $\chi_{d,\beta}$ is defined as
    $
    \chi_{d,\beta} 
    = \begin{cases}
    0\,, &\beta \leq d\,, \\
    1\,, &\beta > d\,.
    \end{cases}
    $

To further illustrate the differences between the functionals $\cG$, $\wt{\cG}_\delta$, and $\cG_{\beta>p}$, we present the following example:
let 
$\cG_{\ell}(u) = \int_{\Omega} \ell(u(\bx)) \, \rmd \bx$ where 
$\ell : \bbR \to \bbR$ be a continuous (but not necessarily convex) function that satisfies
\begin{equation*}
    c(1 - |u|^m) \leq \ell(u) \leq C (1+|u|^m)\,, \quad \text{ for some } m \in [1,\infty)\,.
\end{equation*}
Then we note the following 
properties and
weak continuity on the spaces $\frak{W}^{\beta,p}[\delta;q](\Omega)$:
\begin{enumerate}
    \item[i)] If $m \in [1,\frac{dp}{d-p})$ for $p < d$, or if $m \in [1,\infty)$ for $p \geq d$,
    then \eqref{eq:LowerOrderTerm} is satisfied, and the functional $ \cG(K_\delta u) = \cG_{\ell}(K_\delta u)$ is well-defined and weakly continuous.
    \item[ii)] If $\beta \leq d$, \eqref{eq:LowerOrderTerm3} is satisfied if $m \in [1,p]$, and $\wt{\cG}(u) = \cG_\ell(u)$ is weakly continuous if $\ell$ is convex.
    \item[iii)] If $\beta > d$, both \eqref{eq:LowerOrderTerm3} and \eqref{eq:LowerOrderTerm2} are satisfied if $m \in [1,p^*_\beta)$,  so either $\wt{\cG}(u)=\cG_\ell(u)$ or $\cG_{\beta>d}(u) = \cG_\ell(u)$ is weakly continuous.
\end{enumerate}

\begin{proof}[Proof of \Cref{thm:WellPosedness:Dirichlet}]
    The proof follows direct methods. 
    First, by \eqref{eq:ConvEst:Deriv:Cor},
    \eqref{eq:LowerOrderTerm}, and \eqref{eq:LowerOrderTerm3}
    \begin{equation*}
        \cG(K_\delta u) + \wt{\cG}(u) \geq - c(1+\Vnorm{u}_{\frak{W}^{\beta,p}[\delta;q](\Omega)}^{\theta p})\,,
    \end{equation*}
    and then by \eqref{eq:LowerOrderTerm2} and the continuous, compact embedding $\frak{W}^{\beta,p}[\delta;q](\Omega) \hookrightarrow L^m(\Omega)$ of \Cref{thm:Compactness:FixedSpace}
    \begin{equation*}
        \cG_{\beta>d}(u) \geq - c(1+\Vnorm{u}_{\frak{W}^{\beta,p}[\delta;q](\Omega)}^{\theta p})\,.
    \end{equation*}
    Therefore,
    \begin{equation}\label{eq:VarProf:Dirichlet:Pf1}
    \begin{split}
        \cF_\delta(v) 
        &\geq \cE_\delta(v) - \cG(K_\delta v) - \wt{\cG}(v) - \cG_{\beta>d}(v) \\
        &\geq \cE_\delta(v) - C_1 \Vnorm{ v }_{\frak{W}^{\beta,p}[\delta;q](\Omega)}^{\theta p} - C_2\,.
    \end{split}
    \end{equation}
    Next, let $G$ be a $W^{1,p}(\Omega)$-continuous extension of $g$ to all of $\Omega$,
    i.e.\ 
    $$\Vnorm{G}_{W^{1,p}(\Omega)} \leq C \Vnorm{Tg}_{W^{1-1/p,p}(\p \Omega)} \leq C \Vnorm{g}_{W^{1-1/p,p}(\p \Omega_D)}$$
    Then 
    by the Poincar\'e inequality in \Cref{thm:PoincareDirichlet} applied to $u - G$ (valid here thanks to the equivalent characterization of \Cref{thm:TraceZero}),
    estimates similar to those in the proof of \Cref{thm:EnergySpaceIndepOfKernel} give for any $v \in \frak{W}^{\beta,p}_{g, \p \Omega_D}[\delta;q](\Omega)$
    \begin{equation*}
        \Vnorm{v-G}_{L^p(\Omega)}^p \leq C [v-G]_{ \frak{W}^{\beta,p}[\delta;q](\Omega) }^p \leq C \cE_\delta(v - G)\,.
    \end{equation*}
    Therefore by \Cref{cor:Embedding} applied to $G$
    \begin{equation*}
        \Vnorm{ v }_{\frak{W}^{\beta,p}[\delta;q](\Omega)}^p 
        \leq C ( \Vnorm{ v - G }_{\frak{W}^{\beta,p}[\delta;q](\Omega)}^p + \Vnorm{ G }_{\frak{W}^{\beta,p}[\delta;q](\Omega)}^p )
        \leq C ( \cE_\delta(v) + \Vnorm{ g }_{W^{1-1/p,p}(\p \Omega)}^p)
    \end{equation*}
    and so combining this with \eqref{eq:VarProf:Dirichlet:Pf1} gives
    \begin{equation}\label{eq:CoercivityEstimate:Dirichlet}
        \begin{split}
        &\Vnorm{ v }_{\frak{W}^{\beta,p}[\delta;q](\Omega)}^p
        \leq \cF_\delta(v)
        + C ( 1 + \Vnorm{ v }_{\frak{W}^{\beta,p}[\delta;q](\Omega)}^{\theta p} )\,,
        \end{split}
    \end{equation}
    for a constant $C$ independent of $v$. This estimate guarantees that $\min \cF_\delta > - \infty$, and moreover guarantees the uniform $\frak{W}^{\beta,p}[\delta;q](\Omega)$-bound of a minimizing sequence $\{u_n\}_n$. Hence, $\{u_n\}_n$ converges weakly in $\frak{W}^{\beta,p}[\delta;q](\Omega)$ to a function $u$, and by weak continuity of traces $u = g$ in the trace sense on $\p \Omega_D$. 

    By \eqref{eq:ConvEst:Deriv:Cor} $K_\delta u_n \rightharpoonup K_\delta u$ weakly in $W^{1,p}(\Omega)$, so 
    \begin{equation*}
        \cG(K_\delta u) \leq \liminf_{n \to \infty} \cG(K_\delta u_n)\,.
    \end{equation*}
    Next, $\wt{\cG}$ is $\frak{W}^{\beta,p}[\delta;q](\Omega)$-weakly lower semicontinuous by assumption. Finally, $\cG_{\beta>d} \equiv 0$ for $\beta \leq d$, and when $\beta > d$, we have that $u_n \to u$ strongly in $L^m(\Omega)$ by \Cref{thm:Compactness:FixedSpace}, so since $\cG_{\beta>d}$ is strongly continuous in $L^m(\Omega)$ by assumption,
    \begin{equation*}
        \cG_{\beta>d}(u) = \lim\limits_{n \to \infty} \cG_{\beta>d}(u_n)\,.
    \end{equation*}
    Therefore $\cF_\delta$ is $\frak{W}^{\beta,p}[\delta;q](\Omega)$-weakly lower semicontinuous, and so $u$ is a minimizer of $\cF_\delta$. 
    \end{proof}

\begin{proof}[Proof of \Cref{thm:WellPosedness:Dirichlet:Hom}]
    The proof is exactly the same, noting that $G = 0$.
\end{proof}

\begin{proof}[Proof of \Cref{thm:WellPosedness:Neumann}]
    The proof follows direct methods.
    The estimate \eqref{eq:VarProf:Dirichlet:Pf1}  holds using the same argument, and by the Poincar\'e inequality \Cref{thm:PoincareNeumann} and estimates similar to those in the proof of \Cref{thm:EnergySpaceIndepOfKernel}, we have for any $v \in \mathring{\frak{W}}^{\beta,p}[\delta;q](\Omega)$
    \begin{equation*}
        \Vnorm{v}_{L^p(\Omega)}^p \leq C [v]_{ \frak{W}^{\beta,p}[\delta;q](\Omega) }^p \leq C \cE_\delta(v)\,.
    \end{equation*}
    
    So combining this directly with \eqref{eq:VarProf:Dirichlet:Pf1} gives
    \begin{equation}\label{eq:CoercivityEstimate:Neumann}
        \Vnorm{v}_{\frak{W}^{\beta,p}[\delta;q](\Omega)}^p \leq \cF_\delta(v) +  C ( 1 + \Vnorm{ v }_{\frak{W}^{\beta,p}[\delta;q](\Omega)}^{\theta p} )\,,
    \end{equation}
    for a constant $C$ independent of $v$.
    The rest of the proof follows similarly to that of \Cref{thm:WellPosedness:Dirichlet}.
\end{proof}

\begin{proof}[Proof of \Cref{thm:WellPosedness:Robin}]
    The proof again follows direct methods.
    The same argument used to prove \eqref{eq:VarProf:Dirichlet:Pf1} gives
    \begin{equation}\label{eq:wellposedness:Robin:pf1}
    \begin{split}
        \cF_{\delta}^R(v) \geq & \cE_\delta(v) + \int_{\p \Omega} b|Tv|^p \, \rmd \sigma - C_1 \Vnorm{ v }_{\frak{W}^{\beta,p}[\delta;q](\Omega)}^{\theta p} - C_2\,.
    \end{split}
    \end{equation}
    Next, by the Poincar\'e inequality \Cref{thm:PoincareRobin}, the lower bound on $b$, and estimates similar to those in the proof of \Cref{thm:EnergySpaceIndepOfKernel}, we have for any $v \in \frak{W}^{\beta,p}[\delta;q](\Omega)$
    \begin{equation*}
        \Vnorm{v}_{L^p(\Omega)}^p \leq C_R \left( [v]_{ 
\frak{W}^{\beta,p}[\delta;q](\Omega) }^p + b_0 \int_{\p \Omega_R} |Tv|^p \, \rmd \sigma \right) \leq C \left( \cE_\delta(v) + \int_{\p \Omega} b|Tv|^p \, \rmd \sigma \right)\,.
    \end{equation*}
    Combining this with \eqref{eq:wellposedness:Robin:pf1}, we get that for a constant $C$ independent of $v$,
    \begin{equation}\label{eq:CoercivityEstimate:Robin}
        \Vnorm{v}_{\frak{W}^{\beta,p}[\delta;q](\Omega)}^p \leq \cF_\delta^R(v) +  C ( 1 + \Vnorm{ v }_{\frak{W}^{\beta,p}[\delta;q](\Omega)}^{\theta p} )\,, 
    \end{equation}
     The rest of the proof is similar to that of the previous arguments, noting that $\int_{\p \Omega} b|Tu|^p \, \rmd \sigma$ is $\frak{W}^{\beta,p}[\delta;q](\Omega)$-weakly lower semicontinuous.
\end{proof}

\section{Local limit}\label{sec:loclim}
The following lemma and its corollary will be central in calculating the local limit as the bulk horizon parameter $\delta$ approaches $0$.

\begin{theorem}\label{thm:LocalizationOfEnergy}
    Let $\cE_\delta$ be as in \eqref{eq:Intro:Varprob:Energy}
    and $\cE_0(u)$ be as in \eqref{eq:LocalizedEnergyDefn}, with all the associated assumptions of their definitions.
    Then $\lim\limits_{\delta \to 0} \cE_\delta(u) = \cE_0(u)$ for all $u \in W^{1,p}(\Omega)$,  and $\lim\limits_{\delta \to 0}  \cE_\delta(u) = + \infty$ if $u \in L^p(\Omega) \setminus W^{1,p}(\Omega)$.
    Moreover, if a sequence $\{u_\delta\}_\delta$ converges to $u$ in $C^2(\overline{V})$ for any $V \Subset \Omega$ as $\delta \to 0$, then
    \begin{equation*}
    \begin{split}
        \lim\limits_{\delta \to 0} &\int_{V} \int_{V} \rho \left( \frac{|\bx-\by|}{\eta_\delta(\bx)} \right) \frac{ \Phi( \frac{ |u_\delta(\bx)-u_\delta(\by)| }{|\bx-\by|} ) }{ |\bx-\by|^{\beta-p} \eta_\delta(\bx)^{d+p-\beta} } 
        \, \rmd \by \, \rmd \bx \\
        &= \bar{\rho}_{p,\beta} \int_{V} \fint_{\bbS^{d-1}} \Phi( |\grad u(\bx) \cdot \bsomega|) \, \rmd \sigma(\bsomega) \, \rmd \bx\,.
    \end{split}
    \end{equation*}
\end{theorem}

\begin{proof} 
    The proof follows exactly the same steps as \cite[Proposition 4.1, Remarks 4.1 and 4.2]{ponce2004new}, just as in the proof of \Cref{thm:LocalizationOfSeminorm}.
\end{proof}

For example, if $\Phi(t) = \frac{t^p}{p}$, and if
$\bar{\rho}_{p,\beta} = \overline{C}_{d,p}$,
then $\cE_0(u) = \frac{1}{p} \int_\Omega |\grad u(\bx)|^p \, \rmd \bx$.

\subsection{Dirichlet Constraint}

We extend the functional $\cF_\delta$, defined for this problem on $\frak{W}^{\beta,p}_{g, \p \Omega_D}[\delta;q](\Omega)$, to a functional $\cF_\delta^D$ defined on all of $L^p(\Omega)$ by setting
\begin{equation}\label{eq:Fxnal:Ext:Dirichlet}
    \overline{\cF}_\delta^D(u) :=
    \begin{cases}
        \cF_\delta(u)\,, & \text{ for } u \in \frak{W}^{\beta,p}_{g, \p \Omega_D}[\delta;q](\Omega)\,, \\
        + \infty\,, & \text{ for } u \in L^p(\Omega) \setminus \frak{W}^{\beta,p}_{g, \p \Omega_D}[\delta;q](\Omega)\,.
    \end{cases}
\end{equation}

\begin{proposition}\label{prop:GammaLimit:Dirichlet}
    With all the assumptions of \Cref{thm:LocLimit:Dirichlet}, define
    \begin{equation}
        \overline{\cF}_0^D(u) :=
        \begin{cases}
        \cF_0(u)\,, & \text{ for } u \in W^{1,p}_{g, \p \Omega_D}(\Omega)\,, \\
        + \infty\,, & \text{ for } u \in L^p(\Omega) \setminus W^{1,p}_{g, \p \Omega_D}(\Omega)\,.
    \end{cases}
    \end{equation}
    Then we have
    \begin{equation}
        \overline{\cF}_0^D(u) = \Gammalim_{\delta \to 0} \overline{\cF}_\delta^D(u)\,,
    \end{equation}
    where the $\Gamma$-limit is computed with respect to the topology of strong convergence on $L^p(\Omega)$.
\end{proposition}

\begin{proof}
    We proceed in two steps. First, we prove that
    \begin{equation}\label{eq:GammaLiminf:Dirichlet}
        \overline{\cF}_0^D(u) \leq \liminf_{\delta \to 0} \overline{\cF}_\delta^D(u_\delta)\,, 
    \end{equation}
    for any sequence $\{ u_\delta \}_\delta \subset L^p(\Omega)$ that converges strongly in $L^p(\Omega)$ to $u$.
    If the right-hand side is $\infty$ then there is nothing to show, so assume that $\liminf_{\delta \to 0} \overline{\cF}_\delta^D(u_\delta) < \infty$.
    If this is the case, then it follows from the estimate \eqref{eq:CoercivityEstimate:Dirichlet}
    (note that $C$ is independent of $\delta$ if $\wt{\cG}_\delta = \wt{\cG}$ satisfies \eqref{eq:LocalLimit:LOTAssump}) and from \Cref{thm:Compactness} that 
    $u \in W^{1,p}(\Omega)$. Further, by the identity $g = T u_\delta = T K_\delta u_\delta$ on $\p \Omega_D$ for all $\delta > 0$ and from the weak $W^{1,p}$-continuity of traces, an application of \Cref{lma:GammaConv:WeakConvOfMoll} gives that $T u = g$ on $\p \Omega_D$. Therefore $\overline{\cF}_0^D(u) < \infty$, and we just need to show that
    \begin{equation}\label{eq:GammaLiminf:Finite:Dirichlet}
        \cF_0^D(u) \leq \liminf_{\delta \to 0} \cF_\delta^D(u_\delta)\,.
    \end{equation}
    To this end, an argument similar to the one used to prove \Cref{thm:Compactness} gives
    \begin{equation}\label{eq:GammaLiminf:Finite:Energy}
        \cE_0(u) \leq \liminf_{\delta \to 0} \cE_\delta(u_\delta)\,,
    \end{equation}
    but with \Cref{thm:LocalizationOfEnergy} used in place of \Cref{thm:LocalizationOfSeminorm}.
    Now, by \Cref{lma:GammaConv:WeakConvOfMoll}
    \begin{equation}\label{eq:GammaLiminf:Finite:LOT}
    \lim\limits_{\delta \to 0} \cG(K_\delta u_\delta) = \cG(u)\,,
    \end{equation}
    since $\cG$ is $W^{1,p}(\Omega)$-weakly continuous.
    Thanks to \eqref{eq:LocalLimit:LOTAssump} and the continuity assumption on $\cG_{\beta>d}$ we additionally have 
    \begin{equation}\label{eq:GammaLiminf:Finite:LOT2}
        \lim\limits_{\delta \to 0} \wt{\cG}(u_\delta) = \wt{\cG}(u) \qquad \text{ and } \qquad \lim\limits_{\delta \to 0} \cG_{\beta>d}(u_\delta) = \cG_{\beta>d}(u)\,.
    \end{equation}
    Therefore \eqref{eq:GammaLiminf:Finite:Energy}, \eqref{eq:GammaLiminf:Finite:LOT}, and \eqref{eq:GammaLiminf:Finite:LOT2} establish \eqref{eq:GammaLiminf:Finite:Dirichlet}, i.e. \eqref{eq:GammaLiminf:Dirichlet} is proved.

    Second, we note that the constant sequence $\{u_\delta\}_\delta = u \in L^p(\Omega)$ serves as a recovery sequence:
    \begin{equation}\label{eq:GammaLimsup:Dirichlet}
        \overline{\cF}_0^D(u) = \lim\limits_{\delta \to 0} \overline{\cF}_\delta^D(u)\,.
    \end{equation}
    This follows from \eqref{eq:CoercivityEstimate:Dirichlet} and \Cref{thm:LocalizationOfEnergy}, along with \Cref{thm:ConvergenceOfConv:W1p} which shows that 
    $\lim\limits_{\delta \to 0} \cG(K_\delta u) = \cG(u)$.
    
    Together \eqref{eq:GammaLiminf:Dirichlet} and \eqref{eq:GammaLimsup:Dirichlet} conclude the proof.    
\end{proof}

\begin{proof}[proof of \Cref{thm:LocLimit:Dirichlet}]
    The result follows from the framework described in \cite[Theorem 1.21]{braides2002gamma}. By the $\Gamma$-limit computation in \Cref{prop:GammaLimit:Dirichlet}, it suffices to show that $\{\cF_\delta^D(u_\delta)\}_\delta$ is equi-coercive in the strong $L^p(\Omega)$ topology, i.e. that $\{u_\delta\}_\delta$ is precompact in the strong $L^p(\Omega)$ topology. But, this follows by noting that the constant $C$ appearing in  \eqref{eq:CoercivityEstimate:Dirichlet} is independent of $\delta$, 
    permitting us to apply the compactness result \Cref{thm:Compactness}.
        
    The case $g =0$ follows the same same setup and steps.
\end{proof}

\subsection{Neumann and Robin Constraints}
\begin{proof}[proof of \Cref{thm:LocLimit:Neumann}]
Similar to the Dirichlet case, we may extend the functional $\cF_\delta$, now defined  on $\mathring{\frak{W}}^{\beta,p}[\delta;q](\Omega)$, 
to a functional $\overline{\cF}_\delta^N$  by setting
$\overline{\cF}_\delta^N(u) =
        \cF_\delta(u)$  for $u \in \mathring{\frak{W}}^{\beta,p}[\delta;q](\Omega)$ while
  $ \overline{\cF}_\delta^N(u) =  
        + \infty$ for $u \in L^p(\Omega) \setminus \mathring{\frak{W}}^{\beta,p}[\delta;q](\Omega)$.
Likewise, with all the assumptions of \Cref{thm:LocLimit:Neumann}, we can extend
 $\cF_0(u)$ on $\mathring{W}^{1,p}(\Omega)$ by
defining
       $ \overline{\cF}_0^N(u)=
        \cF_0(u)$ for $u \in \mathring{W}^{1,p}(\Omega)$, while 
        $\overline{\cF}_0^N(u)=
        + \infty$ for $u \in L^p(\Omega) \setminus \mathring{W}^{1,p}(\Omega)$.
    Then  we can show  that as $\delta\to 0$, $   \overline{\cF}_0^N(u) $ is the $\Gamma$-limit  of
$\overline{\cF}_\delta^N(u)$ 
 with respect to the topology of strong convergence on $L^p(\Omega)$. Indeed, the proof follows the same steps as that of \Cref{prop:GammaLimit:Dirichlet}, with the estimate \eqref{eq:CoercivityEstimate:Neumann} used in place of \eqref{eq:CoercivityEstimate:Dirichlet}, and with
    the additional note that if a sequence $\{u_\delta\}_\delta \subset \mathring{\frak{W}}^{\beta,p}[\delta;q](\Omega)$ converges strongly in $L^p(\Omega)$ to a function $u$, then $(u)_\Omega = 0$, 
    Then the proof can be completed by following the same argument as the proof of \Cref{thm:LocLimit:Dirichlet}.
\end{proof}

\begin{proof}[proof of \Cref{thm:LocLimit:Robin}] 
  By similarly extending 
  $\cF_\delta^R$ 
  and  $\cF_0^R$ 
  to $  
  \overline{\cF}_\delta^R$
  and $  
  \overline{\cF}_0^R$ respectively, we can get a similar conclusion on the $\Gamma$-limit for the Robin case, with \eqref{eq:CoercivityEstimate:Robin} used in place of \eqref{eq:CoercivityEstimate:Dirichlet} or \eqref{eq:CoercivityEstimate:Neumann}.
Then the proof of \Cref{thm:LocLimit:Robin}
follows from argument similar to the previous proofs.
 \end{proof}

\section{Conclusion}

We have presented a study of nonlocal function spaces with heterogeneous localization, and used its features to study associated variational analysis problems.
The scaling of the kernels, and the range of $\beta$, have allowed us to treat simultaneously both fractional and convolution-type problems, with the same class of boundary information.

Additional properties of the function spaces can be recovered in a straightforward way using the analysis contained in this work, including finer embeddings, Hardy inequalities, and characterizations of dual spaces.

We note that the theory presented here applies to general Lipschitz domains. We also treat the case of general orders of differentiability, i.e. $k_q$, and $k_\lambda$, that are associated with the various functions used for localization instead of assuming them to be $\infty$ all the time.
Our primary motivation for this choice is to allow for flexibility of the models in implementation, as we demonstrate with the following scenario. 
First, let $k \geq 2$ be some integer, and suppose that $C^k$-smoothness of the heterogeneous localization $\eta$ is desired, with boundedness on all partial derivatives up to and including order $k$.
If it happens that $\Omega$ is a $C^k$ domain, then the choice of $\lambda = d_{\p \Omega}$ is possible. However, some care must be taken, as $d_{\p \Omega}$ does not belong to $C^{k}(\overline{\Omega})$, but rather there exists $\veps_\Omega > 0$ such that $d_{\p \Omega}$ is $C^k$ on the set $\{ \bx \in \overline{\Omega} : \, d_{\p \Omega}(\bx) \leq \veps_\Omega \}$; see \cite{foote1984regularity}. If $q$ is chosen to satisfy \eqref{assump:NonlinearLocalization} for $k_q = k$ with $q(r)$ constant for $r \geq \veps_\Omega$, it follows that the resulting heterogeneous localization $\eta[d_{\p \Omega},q]$ belongs to $C^k(\overline{\Omega})$.
If it is not the case that $\Omega$ is $C^k$, then one can consider, in place of $d_{\p \Omega}$, a generalized distance $\lambda$ satisfying \eqref{assump:Localization} for some $k_\lambda \geq k+1$. Then $\eta(\bx) = q(\lambda(\bx))$ belongs to $C^k(\Omega)$, but it is not guaranteed that its derivatives remain bounded near $\p \Omega$. In that case one can modify $q$, and choose instead a function $\tilde{q}$ that satisfies \eqref{assump:NonlinearLocalization} for $k_q \geq k+1$, and further satisfies $\tilde{q}'(0) = \ldots = \tilde{q}^{(k)}(0) = 0$.
Then an application of Fa\`a di Bruno's formula shows that 
$|D^\alpha \eta[\lambda,\tilde{q}](\bx)| \leq C d_{\p \Omega}(\bx)^{k+1-|\alpha|}$ for all $\bx \in \Omega$ and for all $|\alpha| \leq k$, where $C$ depends only on $\tilde{q}$, $\alpha$, and $\kappa_\alpha$.

Although the well-posedness of these variational problems   in natural function spaces has a relatively clear picture,
there are a number of fundamental questions that remain to be answered. Establishing suitable regularity properties for the models in this work is important for mathematical theory and physical consistency. At the same time, analysis of this type for generalizations of these models -- for instance nonlocal models with $\Phi$ non-convex such as in \cite{Lipton-2014} -- are worth investigating. Further, one may ask if a nonlocal analog of the Green's identity can be shown for operators that involve heterogeneous localization, so that the variational problems considered in this work can be placed in natural correspondence with a pointwise form as suggested in the example above. 
We will show, in the next paper in this series, 
that different localization strategies result in different forms of the proper nonlocal Green's identity \cite{nonlocalBCII}. Intuitively, the boundary condition for the nonlocal problem will be consistent with the classical boundary condition if the function $\eta_\delta(\bx)$ vanishes at a faster rate than $\dist(\bx, \partial \Omega)$. This calls for further (and more delicate) mathematical analysis and also bears significant consequences in the application of localization strategies to nonlocal modeling.

\section*{Acknowledgements}
The presentation of this work benefited from discussions between the authors and Zhaolong Han, Tadele Mengesha, and Xiaochuan Tian during a Structured Quartet Research Ensemble (SQuaRE) at the American Institute of Mathematics titled
\textit{Variational methods for multiscale and nonlinear nonlocal models with applications to peridynamics}
in San Jose in May 2023. 
The authors thank them for their many insightful and helpful comments, which has improved the scope and presentation of the materials. 
The authors additionally thank X. Tian for her invaluable help in completing a version of the proof of \Cref{thm:Density}.

\bibliographystyle{siamplain}
\bibliography{References2}

\end{document}